\documentclass[article]{amsart}
\usepackage{amssymb,amsfonts,amsmath,amsthm}
\usepackage[all,arc]{xy}
\usepackage{enumerate}
\usepackage{mathrsfs}
\usepackage[toc,page]{appendix}
\usepackage[left=3cm, right=3cm, bottom=3cm]{geometry}
\usepackage{graphicx}
\usepackage{tabularx}
\usepackage{multirow} % multirow in table
\usepackage{caption} % table caption with no number
\usepackage{url}
\usepackage{color}
\usepackage{tikz-cd}  % graph
\usepackage{mathdots} %\ddots, \vdots, \iddots
\usepackage{romannum} %Roman number
\usepackage{hyperref} %hyperlink for reference
\usepackage{float} %table follow paragraphp

\newtheorem{thm}{Theorem}[section]

\newtheorem*{thm*}{Theorem}
\newtheorem*{cor*}{Corollary}
\newtheorem*{prop*}{Proposition}
\newtheorem{cor}[thm]{Corollary}
\newtheorem{prop}[thm]{Proposition}
\newtheorem{lem}[thm]{Lemma}

\theoremstyle{definition}
\newtheorem{defn}[thm]{Definition}

\newtheorem*{conv*}{Convention}

\newtheorem*{notn*}{Notation}

\theoremstyle{remark}
\newtheorem{rem}[thm]{Remark}

\newtheorem*{termi*}{\textbf{Terminology}}

\newtheorem*{idea*}{Idea}

\newcommand{\Spec}{{\rm Spec}}

\makeatletter
\let\c@equation\c@thm
\makeatother
\numberwithin{thm}{section}
\numberwithin{equation}{section}

\title[Filtered Stokes G-local Systems]{Filtered Stokes G-local Systems in Nonabelian Hodge Theory on Curves}

\author{Pengfei Huang and Hao Sun}

\begin{document}
\pagenumbering{arabic}
\maketitle
\begin{abstract}
In the wild nonabelian Hodge correspondence on curves, filtered Stokes $G$-local systems are regarded as the objects on the Betti side. In this paper, we demonstrate a construction of the moduli space of them, called the Betti moduli space, and it reduces to the wild character variety when the Betti weights are trivial. We study some particular examples including Eguchi–-Hanson space and the Airy equation together with the corresponding moduli spaces. Furthermore, we provide a proof of the correspondence among irregular singular $G$-connections, Stokes $G$-local systems, and Stokes $G$-representations. This correspondence can be  viewed as the $G$-version of irregular Riemann--Hilbert correspondence on curves.
\end{abstract}
	
\flushbottom
	
%\tableofcontents
%\newpage
	
\renewcommand{\thefootnote}{\fnsymbol{footnote}}
	
\section{Introduction}

\subsection{Background}
The study of the nonabelian Hodge correspondence on noncompact curves begins with Simpson \cite{Sim90}, where he introduced filtered regular Higgs bundles, filtered regular $D_X$-modules, and filtered local systems along with corresponding stability conditions to establish a one-to-one correspondence among them. This correspondence is famously known as the tame nonabelian Hodge correspondence, where tameness characterizes a polynomial growth condition of flat sections or the regularity of meromorphic connections. Under this framework, filtered local systems are local systems on a noncompact curve with additional structures called weighted filtrations (also called parabolic structures). They correspond to fundamental group representations satisfying certain compatibility conditions determined by weights (these conditions are trivial in some sense). The existence of weights makes the stability of filtered local systems not equivalent to the irreducibility of the corresponding fundamental group representations. In a similar way, Biquard--Boalch \cite{BB04} generalized Simpson's correspondence to a broader framework beyond tameness, called wildness, where meromorphic connection can be irregular. They established a one-to-one correspondence between (poly)stable filtered irregular Higgs bundles and  (poly)stable filtered irregular $D_X$-modules under a ``very good'' condition. Although this work is known as the (unramified) wild nonabelian Hodge correspondence, it does not touch the objects from the Betti side.

To establish a comprehensive wild nonabelian Hodge correspondence, a crucial step involves figuring out the correct objects on the Betti side, which not only correspond to filtered irregular $D_X$-modules but also preserves stability conditions. Classically, the Riemann--Hilbert correspondence connects $D_X$-modules with local systems. The same idea holds in the wild case and it is well-known that there is a one-to-one correspondence between connections with irregular singularities (or called irregular $D_X$-modules) and Stokes local systems \cite{Sib77,Mal78,Mal83,Lod94}, which is also known as the irregular Riemann--Hilbert correspondence. Moreover, this correspondence was studied systematically in a great generality by Boalch for reductive groups as the structure group \cite{Boa14,Boa18,Boa21}. Inspired by previous works, the authors introduced (Betti) weights to Stokes local systems, which are called filtered Stokes local systems and regarded as the objects on the Betti side, and define its stability condition. Then, an (unramified) wild nonabelian Hodge correspondence at the level of categories was established \cite{HS22}. Furthermore, such a correspondence also holds for complex reductive groups as the structure groups. For the case of trivial Betti weights, filtered Stokes local systems from the Betti side reduce to Stokes local systems. In this case, Boalch constructed the moduli space of Stokes local systems by identifying Stokes local systems with representations of the fundamental groupoid of irregular curves, which is now known as the wild character variety. However, as the tame case Simpson ever observed, the stability of filtered Stokes $G$-local systems no longer aligns with the irreducibility of the corresponding representations, which is different from the case of trivial Betti weights investigated by Boalch--Yamakawa \cite{BY23}. Consequently, this leads to the fact that a construction of the moduli spaces of filtered Stokes $G$-local systems becomes challenging.

In this paper, we provide a construction of this moduli space in detail, for both unramified and ramified cases and study some specific examples. Moreover, the (unramified) wild nonabelian Hodge correspondence considered in \cite{HS22} holds at the level of moduli spaces.

\subsection{Main result}
We introduce some notations first. Let $G$ be a connected complex reductive group with a given maximal torus $T$. Let $X$ be a connected smooth projective algebraic curve over $\mathbb{C}$ with a collection $\boldsymbol{D}$ of finite points. Denote by $X_{\boldsymbol{D}} := X \backslash \boldsymbol{D}$ the punctured curve. We also fix a collection of irregular types $\boldsymbol{Q} = \{Q_x, x \in \boldsymbol{D}\}$ and a collection of weights $\boldsymbol\theta =\{\theta_x , x \in \boldsymbol{D}\}$ labelled by points in $\boldsymbol{D}$, where a weight is regarded as a rational cocharacter ${\rm Hom}(\mathbb{G}_m, T) \otimes_{\mathbb{Z}} \mathbb{Q}$.

To construct the moduli space of filtered Stokes $G$-local systems, we first relate Stokes $G$-local systems to Stokes $G$-representations, and study the irregular Riemann--Hilbert correspondence for reductive groups $G$ on $X_{\boldsymbol{D}}$ in both unramified and ramified cases in \S\ref{sect_irre_RH_corr} (Theorem \ref{thm_global_con_sys} and \ref{thm_boa14A.3}). The proof is similar to \cite[Appendix]{Boa14}, where the author studied the unramified case. Moreover, Hohl and Jakob recently proved the correspondence (Theorem \ref{thm_global_con_sys}) in a different way via Tannakian categories and we refer the reader to \cite[\S 3]{HJ24} for more details.

In \S\ref{sect_moduli}, we review the stability conditions on filtered Stokes $G$-local systems (Definition \ref{defn_stab_cond_G_Stokes}), which is the stability condition considered in the wild nonabelian Hodge correspondence \cite{HS22}, and construct the moduli space of filtered Stokes $G$-local systems. We would like to point out that King's result \cite{King94} inspires us to relate the stability condition of filtered Stokes $G$-local systems to the stability condition in the sense of GIT (Proposition \ref{prop_stab_equiv}), which help us to construct the moduli sapce.
\begin{thm}[Theorem \ref{thm_Sto_moduli}]
The moduli space $\mathcal{M}_{\rm B}(X_{\boldsymbol{D}},G,\boldsymbol{Q},\boldsymbol\theta)$ of degree zero $\boldsymbol\theta$-filtered Stokes $G$-local systems with irregular type $\boldsymbol{Q}$ on $X_{\boldsymbol{D}}$ exists as a quasi-projective variety.
\end{thm}

In \S\ref{sect_eg}, we study some examples of filtered Stokes $G$-local systems together with their moduli spaces. Here is a summary of the results.
\begin{enumerate}
    \item When the weights $\boldsymbol\theta$ are trivial, the moduli space $\mathcal{M}_{\rm B}(X_{\boldsymbol{D}},G,\boldsymbol{Q},\boldsymbol\theta)$ is exactly the wild character variety, which has been studied in \cite{Boa14,BY15,BY23,DDP18,HMW19}.
    \item When the irregular types $\boldsymbol{Q}$ are trivial, the moduli space $\mathcal{M}_{\rm B}(X_{\boldsymbol{D}},G,\boldsymbol{Q},\boldsymbol\theta)$ is the moduli space of filtered $G$-local systems \cite{HS23}, which is the Betti moduli space in the tame nonabelian Hodge correspondence.
    \item We equip Eguchi--Hanson space with distinct weights and find a filtered Stokes local system that is stable but not irreducible. This example shows that the Betti moduli spaces in the wild nonabelian Hodge correspondence may not be the wild character varieties.
    \item We calculate the corresponding Stokes $G$-representations of the classical Airy equation and show that the moduli space, where it lies, is a single point. Therefore, the corresponding irregular singular connection of the Airy equation is both rigid and physically rigid. This result is also obtained in \cite[Theorem 1.2.1]{HJ24} recently. Moreover, we prove that in the case of $G={\rm SL}_2(\mathbb{C})$ and a ramified irregular type $Q$, the moduli space $\mathcal{M}_{\rm B}(X_{\boldsymbol{D}}, {\rm SL}_2(\mathbb{C}), Q,\theta)$ is always isomorphic to the wild character variety and does not depend on the weight (Proposition \ref{prop_ram_sl2}). This result can be generalized to $G = {\rm SL}_n(\mathbb{C})$ in a certain extent (Remark \ref{rem_ram_sln}).
\end{enumerate}

In \S\ref{sect_nahc}, we construct the Betti moduli space (Corollary \ref{cor_Betti}) in the (unramified) wild nonabelian Hodge correspondence for principal bundles on curves given by the authors in \cite[Theorem in \S 1]{HS22}, and then we show the correspondence holds at the level of moduli spaces (Theorem \ref{thm_nahc}). Moreover, the corresponding moduli space is shown admitting a hyperK\"alher structure (Theorem \ref{hK-Higgs}).

\vspace{2mm}

\textbf{Acknowledgments}.
The authors would like to thank Konstantin Jakob, Yichen Qin and Xiaomeng Xu for helpful discussions. Pengfei Huang acknowledges funding from the European Research Council (ERC) under the European Union’s Horizon 2020 research and innovation program (grant agreement No 101018839) and Deutsche Forschungsgemeinschaft (DFG,
Projektnummer 547382045). Hao Sun is partially supported by National Key R$\&$D Program of China (No. 2022YFA1006600).
\vspace{2mm}

\section{Irregular Riemann--Hilbert Correspondence}\label{sect_irre_RH_corr}

Let $X$ be a connected smooth projective algebraic curve over $\mathbb{C}$ (i.e. a connected compact Riemann surface) with a collection of finite points $\boldsymbol{D}$. Denote by $X_{\boldsymbol{D}} := X \backslash \boldsymbol{D}$ the punctured curve. In this section, we study the irregular Riemann--Hilbert correspondence for connected complex reductive groups $G$ on $X_{\boldsymbol{D}}$ in both unramified and ramified cases. Fixing a collection of irregular types $\boldsymbol{Q} = \{Q_x, x \in \boldsymbol{D}\}$, we prove the equivalence of the following three categories on $X_{\boldsymbol{D}}$
\begin{itemize}
    \item the category of $G$-connections with irregular type $\boldsymbol{Q}$;
    \item the category of Stokes $G$-local systems with irregular type $\boldsymbol{Q}$;
    \item the category of Stokes $G$-representations with irregular type $\boldsymbol{Q}$.
\end{itemize}
In \S\ref{subsect_lo_corr}, we first study the correspondence on a punctured disc $\mathbb{D}^*$, and then in \S\ref{subsect_glo_corr}, we prove the equivalence of categories on $X_{\boldsymbol{D}}$. In \S\ref{subsect_Sto_G_rep}, we give an equivalent description of the space of Stokes $G$-representations, which will be used to construct the moduli space of filtered Stokes $G$-local systems in \S\ref{sect_moduli}.

\subsection{Local Correspondence}\label{subsect_lo_corr}
We fix some notations first
\begin{align*}
    & R = \mathbb{C}\{z\}, \quad \quad \widehat{R} = \mathbb{C}[\![z]\!], \\
    & K = \mathbb{C}(\!\{z\}\!), \quad \widehat{K} = \mathbb{C}(\!(z)\!).
\end{align*}
Sometimes we add subscript `z' or `w' to emphasize the local coordinate, for instance, $R_z = \mathbb{C} \{z\}$ and $R_w = \mathbb{C} \{w\}$. Let $\mathbb{D}= \Spec \, R$ and $\widehat{\mathbb{D}}= \Spec \, \widehat{R}$ be the disc and formal disc, respectively, and then
\begin{align*}
    & \mathbb{D}^{\ast}: \text{ punctured disc}, \\
    & \widehat{\mathbb{D}}^{\ast}: \text{ formal punctured disc}.
\end{align*}
Let $G$ be a connected complex reductive group with a given maximal torus $T$. Let $\mathfrak{g}$ (resp. $\mathfrak{t}$) be the Lie algebra of $G$ (resp. $T$). Denote by $\mathcal{R}$ the set of roots.

Let $V$ be a $G$-bundle on $\mathbb{D}^{\ast}$ with a $G$-connection $\nabla$ and a $G$-connection in this paper is always assumed to be algebraic. To simplify the terminology, such a pair $(V,\nabla)$ is also called a $G$-connection. With respect to the local coordinate `z', we write $\nabla$ as
\begin{align*}
    \nabla = d + A(z)dz,
\end{align*}
where $d$ is the exterior differential and $A(z) \in \mathfrak{g}(K)$ is the connection form. Note that the connection form $A(z)$ of $\nabla$ is in $\mathfrak{g}(K)$ because we are working on connections with irregular singularities. Two $G$-connections $\nabla_1= d + A_1(z)dz$ and $\nabla_2 = d+ A_2(z)dz$ are \emph{(gauge) equivalent} if there exists $g \in G(K)$ such that $g \circ \nabla_1 = \nabla_2$, i.e.
\begin{align*}
    A_2(z) = - g' \cdot g^{-1} + g A_1(z) g^{-1}.
\end{align*}
Moreover, we say that $\nabla_1$ and $\nabla_2$ are \emph{formally (gauge) equivalent} if there exists an element $g \in G(\widehat{K})$ such that $g \circ \nabla_1 = \nabla_2$.

Now given a $G$-connection $\nabla = d + A(z)dz$, we consider the following set
\begin{align*}
    \widehat{G}(\nabla) := \{ g \in G(\widehat{K}) \, | \, g \circ \nabla \emph{ is a $G$-connection}  \}.
\end{align*}
Equivalently, this set can be regarded as
\begin{align*}
    \widehat{G}(\nabla) = \{ g \in G(\widehat{K}) \, | \, - g' \cdot g^{-1} + g A(z)g^{-1} \in \mathfrak{g}(K) \}.
\end{align*}
Clearly,
\begin{align*}
    G(K) \subseteq \widehat{G}(\nabla) \subseteq G(\widehat{K}).
\end{align*}
The quotient set $\widehat{G}(\nabla) / G(K)$ classifies all $G$-connections which are formally equivalent to $\nabla$. Moreover, taking an arbitrary element $\nabla' \in \widehat{G}(\nabla)$, we have $\widehat{G}(\nabla) = \widehat{G}(\nabla')$, and thus $\widehat{G}(\nabla) / G(K) = \widehat{G}(\nabla') / G(K)$, which implies that the quotient does not depend on the choice of representatives.

When $G = {\rm GL}_n(\mathbb{C})$, the Malgrange--Sibuya isomorphism theorem \cite{Mal78,Mal83,Sib77} describes $\widehat{G}(\nabla) / G(K)$ as a non-abelian cohomological set. The argument can be applied to complex reductive groups in the same way. We briefly review the setup and state the result. For convenience, let $\mathbb{D} = \mathbb{C}$ and denote by $\widetilde{\mathbb{D}} \rightarrow \mathbb{D}$ the real oriented blowup of $\mathbb{D}$ at $0 \in \mathbb{D}$, i.e. $\widetilde{\mathbb{D}} = S^1 \times [0,+ \infty )$. Let $U$ be an open subset of $S^1$, and we define an open subset
\begin{align*}
    \widetilde{U} = \{ (\rho, \theta) \in \widetilde{\mathbb{D}} \, | \, \rho >0 , \,  \theta \in U\}
\end{align*}
of $\widetilde{\mathbb{D}}$. Now, we define a nonabelian sheaf $\Lambda_\nabla$ on $S^1$, of which a germ $g$ at $\theta \in S^1$ is an (holomorphic) element in $G(\mathcal{O}(\widetilde{U}))$ such that
\begin{itemize}
    \item $g$ is asymptotic to the identity on $\widetilde{U}$ at $0$;
    \item $g \circ \nabla = \nabla$.
\end{itemize}
We refer the reader to \cite[\S 3]{Mal83} and \cite[I.2]{Lod94} for more details about the construction of this sheaf. With a similar argument as in the case of $G={\rm GL}_n(\mathbb{C})$, we have the following Malgrange--Sibuya theorem for reductive groups:

\begin{thm}[Theorem A.1 in \cite{Boa14}]\label{thm_boa14A.1}
There exists a bijection between the sets $\widehat{G}(\nabla) / G(K)$ and $H^1(S^1 , \Lambda_\nabla)$.
\end{thm}

To give a more precise description of the cohomology $H^1(S^1 , \Lambda_\nabla)$, we first introduce \emph{irregular types}. An element in the following form
\begin{align*}
    Q(z) = q_{-n} z^{-n/r} +  \dots + q_{-1} z^{-1/r}
\end{align*}
for some positive integers $n$ and $r$, where $q_{-i} \in \mathfrak{t}$, is called an \emph{irregular type}. Under the substitution $z=w^r$, we have
\begin{align*}
    Q(w) = q_{-n} w^{-n} +  \dots + q_{-1} w^{-1}.
\end{align*}
The substitution $z=w^r$ can be regarded as choosing a covering $\mathbb{D} \rightarrow \mathbb{D}$. An irregular type is called \emph{unramified} if $r=1$, and \emph{ramified} if $r \geq 2$.

\begin{defn}
Given an irregular type $Q$, a $G$-connection $\nabla =  d+ A(z)dz$ is with \emph{irregular type $Q$}, if under the substitution $z=w^r$, the connection $d + A(w)d(w^d)$ is formally equivalent (under the action of $G(\widehat{K}_w)$) to a connection in the form
\begin{align*}
    d + dQ + b_{-1} \frac{dw}{w}
\end{align*}
such that $[Q,b_{-1}] = 0$, where $b_{-1} \in \mathfrak{g}$.
\end{defn}
It has been proven that any $G$-connection is of a certain irregular type.
\begin{thm}[\cite{BV83}]
    Given any $G$-connection $d + A(z)dz$, where $A(z) = a_{-n} z^{-n} + \dots \in \mathfrak{g}(K_z)$, there exists a positive integer $r$ such that under the substitution $z = w^r$, the $G$-connection $d + A(w)dw$ is formally equivalent to $d + B(w)dw$ (under the action of $G(\widehat{K}_w)$), where $B(w) = b_{-n'} w^{-n'} + \dots \in \mathfrak{g}(K_w)$ such that
    \begin{itemize}
        \item $b_{i} \in \mathfrak{t}$ for $i \leq -2$;
        \item $b_i =0$ for $i \geq 0$;
        \item $[b_i, b_j] = 0$.
    \end{itemize}
\end{thm}

\begin{rem}
In the above theorem, although the integer $r$ is not unique, in this paper we always assume that the integer $r$ we choose is the smallest one.
\end{rem}

Now we fix a $G$-connection $\nabla = d + A(z)dz$ with irregular type $Q$. In the case of $G = {\rm GL}_n(\mathbb{C})$, Loday-Richaud gives a constructive description of $H^1(S^1, \Lambda_\nabla)$ from unipotent Lie groups \cite{Lod94}. The arguments can be applied to complex reductive groups as well. In the following, we only give the statement and refer the reader to \cite{Boa14,BY15} for the description.

Given a root $\alpha \in \mathcal{R}$, it determines a meromorphic function $q_\alpha(z) : = \alpha(Q(z))$. A direction $d \in S^1$ is an \emph{anti-Stokes direction} (supported by $\alpha$) if the meromorphic function $\exp(q_\alpha(z))$ has maximal decay as $z$ goes to zero in the direction. Denote by $\mathbb{A}$ the set of all anti-Stokes directions with respect to the given irregular type $Q$. Given an anti-Stokes direction $d \in \mathbb{A}$, let $\mathcal{R}(d) \subseteq \mathcal{R}$ be the subset of roots supporting $d$, and for each $\alpha\in\mathcal{R}$, let $U_\alpha:=\exp(\mathfrak{g}_\alpha)\subseteq G$ be the corresponding unipotent subgroup. Denote by $\mathbb{S}{\rm to}_d$ the image of the product map $\prod_{\alpha \in \mathcal{R}(d)} U_{\alpha} \rightarrow G$, which is a unipotent group \cite[I.4.8]{Lod94}. We define
\begin{align*}
    \mathbb{S}{\rm to}(Q):= \prod_{d \in \mathbb{A}} \mathbb{S}{\rm to}_d.
\end{align*}

\begin{thm}[Theorem A.2 in \cite{Boa14}]\label{thm_boa14A.2}
    There is a bijection
    \begin{align*}
        \mathbb{S}{\rm to}(Q) \rightarrow H^1(S^1 , \Lambda_\nabla ).
    \end{align*}
\end{thm}

\begin{rem}
In \cite[Appendix]{Boa14}, although the author only deals with the unramified case, the results and arguments hold for the ramified case. Therefore, we only state Theorem \ref{thm_boa14A.1} and Theorem \ref{thm_boa14A.2} without a proof.
\end{rem}

In \cite{BY15}, the authors define a local system $\mathcal{I}$ on $S^1$, of which sections over sectors are functions in the form $Q=\sum_{i=1}^n q_{-i} z^{-i/r}$, where $q_{-i/r} \in \mathbb{C}$ and $n,r \in \mathbb{N}$. The sheaf $\mathcal{I}$ can be regarded as a vast disjoint union of circle coverings of $S^1$, and each component of $\mathcal{I}$ (i.e. an element in $\pi_0(\mathcal{I})$) is a covering of $S^1$. Given a point $p \in S^1$, the fiber $\mathcal{I}_p$ is a free $\mathbb{Z}$-module. We define a pro-tori
$\mathcal{T}_p:= {\rm Hom}(\mathcal{I}_p, \mathbb{C}^*)$, where a pro-tori is an inverse limit of torus. We obtain a system $\mathcal{T}$ of pro-tori over $S^1$, of which $\mathcal{T}_p$ is the fiber at $p \in S^1$.

\begin{defn}
An \emph{$\mathcal{I}$-graded $G$-local system} on $S^1$ is a $G$-local system $L$ on $S^1$ together with a morphism $\mathcal{T} \rightarrow {\rm Aut}(L)$ of local systems over $S^1$ factoring through an algebraic quotient of $\mathcal{T}$.
\end{defn}

Boalch--Yamakawa proved the following equivalence of categories.

\begin{thm}[Theorem 6 in \cite{BY15}]
The category of $G$-connections on $\widehat{\mathbb{D}}^*$ is equivalent to the category of $\mathcal{I}$-graded $G$-local systems on $S^1$.
\end{thm}

\begin{rem}
Let $L$ be an $\mathcal{I}$-graded $G$-local system on $S^1$, and denote by $\nabla$ the corresponding $G$-connection on $\widehat{\mathbb{D}}^*$. We briefly state how irregular types of $\nabla$ corresponds to morphisms $\mathcal{T} \rightarrow {\rm Aut}(L)$. We fix a point $p \in S^1$, and then we obtain a morphism $\mathcal{T}_p \rightarrow {\rm Aut}(L_p) \cong G$ by restricting to $p$. Moreover, we suppose that the image of the morphism lies in the maximal torus $T$ of $G$. Since the image of $\mathcal{T}_p \rightarrow T$ is a quotient of $\mathcal{T}_p$, it is equivalent to a finite rank free $\mathbb{Z}$-submodule of $\mathcal{I}_p$, of which the generators can be regarded as irregular types. In this sense, we say that a $\mathcal{I}$-graded $G$-local system is of \emph{irregular type} $Q$, if the corresponding connection $\nabla$ is with irregular type $Q$. Note that in \cite{BY15}, it is called \emph{irregular classes} of $\mathcal{I}$-graded $G$-local systems, while in this paper, we use the terminology \emph{irregular types} for convenience.
\end{rem}

Recall that $\widetilde{\mathbb{D}}$ is the real oriented blow-up of $\mathbb{D}$ at zero, and the zero point is usually denoted by $x$. Denote by $\partial = S^1$ the boundary circle. We draw a concentric circle (a halo) $\partial'$ on $\widetilde{\mathbb{D}}$, and denote by $\mathbb{H}$ the region between $\partial$ and $\partial'$. In other words, $\mathbb{H}$ is regarded as a tubular neighbourhood of $\partial$ with another boundary circle $\partial'$. We puncture $\partial'$ at $\# \mathbb{A}$ many distinct points and denote them by $\{x_d,\, d \in \mathbb{A}\}$. According to the anti-Stokes directions, we require that all the $\# \mathbb{A}_x$ auxiliary small cilia between each anti-Stokes direction and its nearby puncture do not cross (see the following picture for example).

\begin{center}
    \includegraphics[scale=0.36]{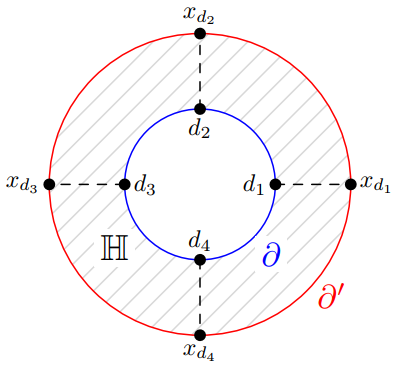}
\end{center}

\noindent Denote by $\mathbb{D}_Q$ the punctured surface obtained in the above way. Moreover, we have
\begin{align*}
    \mathbb{D}_Q \hookrightarrow \widetilde{\mathbb{D}} \rightarrow \mathbb{D}.
\end{align*}

\begin{defn}
    A \emph{Stokes $G$-local system with irregular type $Q$} on $\mathbb{D}^*$ is a $G$-local system $L$ on $\mathbb{D}_Q$ such that $L|_{\mathbb{H}}$ is with irregular type $Q$ and the monodromy around each puncture $x_d$ in $\mathbb{D}_Q$ lies in $\mathbb{S}{\rm to}_d$.
\end{defn}

\begin{rem}
    Since $\mathbb{H}$ is a tubular neighbourhood of $\partial$, the fundamental group of $\mathbb{H}$ is isomorphic to the fundamental group of $\partial$. Then, the category of $G$-local systems on $\mathbb{H}$ is equivalent to the category of $G$-local systems on $\partial$. Thus, a $G$-local system on $\mathbb{H}$ with irregular type $Q$ actually means that the corresponding $G$-local system on $\partial$ is with irregular type $Q$.
\end{rem}

\begin{thm}[Local Correspondence]\label{thm_local_con_sys}
    The category of $G$-connections with irregular type $Q$ on $\mathbb{D}^*$ is equivalent to the category of Stokes $G$-local systems with irregular type $Q$ on $\mathbb{D}^*$.
\end{thm}

\begin{proof}
We fix a $G$-connection $\nabla_0$ with irregular type $Q$. Then, isomorphism classes of $G$-connections with irregular type $Q$ are classified by $\widehat{G}(\nabla_0) / G(K)$. As we discussed above Theorem \ref{thm_boa14A.1}, the quotient set does not depend on the choice of $\nabla_0$. Given an arbitrary $G$-connection $\nabla'$ with irregular type $Q$, denote by $L'$ (resp. $L_0$) the corresponding $G$-local system of $\nabla'$ (resp. $\nabla_0$). Moreover, we regard $L_0$ as a $G$-local system on $\mathbb{H}$, while $L'$ as a $G$-local system on $\mathbb{D}^*$. By Theorem \ref{thm_boa14A.1} and \ref{thm_boa14A.2}, there is a bijection between $\mathbb{S}{\rm to}(Q) = \prod_{d \in \mathbb{A}} \mathbb{S}{\rm to}_d$ and $\widehat{G}(\nabla_0) / G(K)$. Denote by $(\gamma_d)_{d \in \mathbb{A}}$ the corresponding elements of $\nabla'$ in $\prod_{d \in \mathbb{A}} \mathbb{S}{\rm to}_d$, where $\gamma_d \in \mathbb{S}{\rm to}_d$. Then, we glue $L_0$ and $L'$ via $\gamma_d$ around each puncture $x_d$. Thus, we obtain a $G$-local system on $\mathbb{D}_Q$, which is clear a Stokes $G$-local system with irregular type $Q$ on $\mathbb{D}^*$.

    On the other hand, given a Stokes $G$-local system $L$ with irregular type $Q$ on $\mathbb{D}_Q$, we obtain two $G$-local systems $L'$ and $L_0$ by taking restrictions to $\mathbb{D}^*$ and $\mathbb{H}$ respectively.  Under isomorphisms, we suppose that $L_0$ is the $G$-local system given by the $G$-connection $\nabla_0$. Since $L$ is a Stokes $G$-local system with irregular type $Q$, the monodromy $\gamma_d$ around puncture $x_d$ for each $d \in \mathbb{A}$ gives an element $(\gamma_d)_{d \in \mathbb{A}} \in \mathbb{S}{\rm to}(Q)$. Therefore, we obtain a $G$-connection with irregular type $Q$.
\end{proof}

\subsection{Global Correspondence}\label{subsect_glo_corr}
Let $X$ be a connected smooth projective algebraic curve over $\mathbb{C}$. Let $\boldsymbol{D}$ be a given collection of finitely many distinct points on $X$, which is also regarded as a reduced effective divisor on $X$, and denote by $X_{\boldsymbol{D}}:=X \backslash \boldsymbol{D}$ the punctured curve, which is also called a noncompact curve. Let $\widetilde{X}$ be the real oriented blow-up of $X$ at each puncture $x \in \boldsymbol{D}$. It is equivalent to consider that $\widetilde{X}$ is obtained from $X$ by replacing each puncture $x \in \boldsymbol{D}$ by an oriented boundary circle $\partial_x$, of which points are considered to be oriented directions emanating from $x$. Now we equip each puncture $x \in \boldsymbol{D}$ with an irregular type $Q_x$, and denote by $\boldsymbol{Q} = \{Q_x, \, x \in \boldsymbol{D}\}$ the collection with irregular types. For each $x \in \boldsymbol{D}$, let $\mathbb{A}_x$ be the set of anti-Stokes directions of $Q_x$. Then, we draw a concentric circle (a halo) $\partial'_x$ on $\widetilde{X}$ near $\partial_x$. Denote by $\mathbb{H}_x$ the region between $\partial_x$ and $\partial'_x$, which is a tubular neighbourhood of $\partial_x$. Then, we puncture $\partial'_x$ at $\# \mathbb{A}_x$ distinct points according to anti-Stokes directions such that all the $\# \mathbb{A}_x$ auxiliary small cilia between each anti-Stokes direction and its nearby puncture do not cross. Denote by $\{x_d, \, d \in \mathbb{A}_x\}$ the collection of punctures with respect to point $x \in \boldsymbol{D}$. Finally, let $X_{\boldsymbol{Q}} \hookrightarrow \widetilde{X}$ be the punctured surface obtained in the above way, which is called the \emph{irregular curve} given by $\boldsymbol{Q}$.

\begin{defn}
A \emph{Stokes $G$-local system with irregular type $\boldsymbol{Q}$} on $X_{\boldsymbol{D}}$ is a $G$-local system $L$ on $X_{\boldsymbol{Q}}$ such that for each puncture $x \in \boldsymbol{D}$, the restriction $L|_{\mathbb{H}_x}$ is with irregular type $Q_x$ (up to isomorphism) and the monodromy around each puncture $x_d$ lies in the Stokes group $\mathbb{S}{\rm to}_d$ for every $d \in \mathbb{A}_x$.
\end{defn}

\begin{thm}[Global Correspondence]\label{thm_global_con_sys}
The category of $G$-connections with irregular type $\boldsymbol{Q}$ on $V$, is equivalent to the category of Stokes $G$-local systems with irregular type $\boldsymbol{Q}$ on $X_{\boldsymbol{D}}$.
\end{thm}

\begin{proof}
    This is an immediate result of the local correspondence (Theorem \ref{thm_local_con_sys}).
\end{proof}

If we do not fix a specific irregular type, then we have the following correspondence, which is regarded as the $G$-version of the classical irregular Riemmann--Hilbert correspondence of Deligne, Malgrange, Sibuya, Loday-Richaud \cite{Lod94,Mal83,Sib77}. Moreover, Hohl and Jakob give a different proof of Theorem \ref{thm_global_con_sys} via Tannakian categories \cite[\S 3]{HJ24}.

\begin{cor}
    The category of $G$-connections on $X_{\boldsymbol{D}}$ is equivalent to the category of Stokes $G$-local systems on $X_{\boldsymbol{D}}$.
\end{cor}

Now we will introduce the fundamental groupoid of $X_{\boldsymbol{Q}}$ and show that the category of Stokes $G$-local systems with irregular type $\boldsymbol{Q}$ on $X_{\boldsymbol{D}}$ is equivalent to the category of specific $G$-representations of the fundamental groupoid of $X_{\boldsymbol{Q}}$. We first introduce two sets $H(\partial)$ and $H$ determined by an irregular type. Given an irregular type $Q$, denote by $H$ the centralizer of $Q = q_n z^{-n/r} +  \dots + q_{-1} z^{-1/r}$ in $G$. More precisely,
\begin{align*}
    H=\{ k \in G \, | \, [k, q_{-i}] = 0 \text{ for each $i$.}  \}
\end{align*}
Denote by $H(\partial) \subseteq G$ the subset of formal monodromies given by $Q$. In the unramfied case, $H=H(\partial)$. Boalch--Yamakawa proved the following result:
\begin{lem}[Lemma 15 in \cite{BY15}]
The $(H\times H)$-action on $H(\partial)$ via $(h_1,h_2)(m) = h_1 m h_2$ gives a $H$-bitorsor structure on $H(\partial)$.
\end{lem}

Now we fix a base point $b_0 \in X_{\boldsymbol{Q}}$, which is not in the boundary circle $\partial_x$ for each $x \in \boldsymbol{D}$. For each boundary circle $\partial_x$ of $X_{\boldsymbol{Q}}$, we choose a base point $b_x$, and denote by $\boldsymbol{b}:=\{b_0 , b_x, x\in\boldsymbol{D}\}$ the set of base points. Let $\Pi_1(X_{\boldsymbol{Q}},\boldsymbol{b})$ be the fundamental groupoid of $X_{\boldsymbol{Q}}$ with $\boldsymbol{b}$ as the set of base points. Here is an explicit description of generators of $\Pi_1(X_{\boldsymbol{Q}},\boldsymbol{b})$:
\begin{enumerate}
\item $\alpha_1,\beta_1,\cdots,\alpha_g,\beta_g$ are loops based at $b_0$ determined by the genus of $X$;
	
\item for each $x\in\boldsymbol{D}$, the simple closed loop $\gamma_x$ based at $b_x$ goes once around $\partial_x$;

\item for each $d\in\mathbb{A}_x$, the loop $\gamma_{x,d}$ based at $b_x$ goes once around the nearby puncture $x_d$ so that $x_d$ is the only puncture inside $\gamma_{x,d}$;

\item for each base point $b_x$, the simple path $\gamma_{0x}$ connects $b_0$ and $b_x$.
\end{enumerate}
For the relations of $\Pi_1(X_{\boldsymbol{Q}},\boldsymbol{b})$, for each $x \in \boldsymbol{D}$, we define
\begin{equation}\label{eq_rela_Pi}\tag{$\ast$}
    \mu_x = \gamma_{0x}^{-1} \cdot \gamma_x \cdot \left( \prod_{d \in \mathbb{A}_x} \gamma_{x,d} \right) \cdot \gamma_{0x},
\end{equation}
which is a loop based at $b_0$. Then, the relation of $\Pi$ is
\begin{align*}
    \left( \prod_{i=1}^g [\alpha_i, \beta_i] \right) \cdot \left( \prod_{x \in \boldsymbol{D}} \mu_x \right) = {\rm id}.
\end{align*}
\noindent In the above setup, $\gamma_x$ is usually regarded as \emph{formal monodromy}, while $\mu_x$ is regarded as \emph{topological monodromy}. In fact, the definition of $\Pi_1(X_{\boldsymbol{Q}},\boldsymbol{b})$ does not depend on the choice of base points, for which use the notation $\Pi$ as an abbreviation. Denote by $\Omega$ the free group generated by generators of $\Pi$ and there is a natural surjection $\Omega \rightarrow \Pi$.

Let $\mathrm{Hom}(\Pi,G)$ be the space of $G$-representations of $\Pi$. An element (point) $\rho\in\mathrm{Hom}(\Pi,G)$ is called a \emph{Stokes $G$-representation with irregular type $\boldsymbol{Q}$} on $X_{\boldsymbol{D}}$ if for each $x\in\boldsymbol{D}$ and $d\in\mathbb{A}_x$, we have $\rho(\gamma_x)\in H(\partial_x)$, which is the set of formal monodromies given by $Q_x$, and $\rho(\gamma_{x,d})\in\mathbb{S}\mathrm{to}_d$. Denote by $\mathrm{Hom}_{\mathbb{S}}(\Pi,G)$ the space of all Stokes $G$-representations with irregular type $\boldsymbol{Q}$, which is a smooth affine variety. Here is an equivalent description of $\mathrm{Hom}_{\mathbb{S}}(\Pi,G)$ with respect to generators and relations of $\Pi$. For each $x \in \boldsymbol{D}$, we define
\begin{align*}
    \mathcal{A}(Q_x) = H(\partial_x) \times \mathbb{S}{\rm to}(Q_x).
\end{align*}
We consider the closed subvariety
\begin{align*}
    {\rm Hom}_{\mathbb{S}}(\Omega, G) := \left( (G\times G)^g\times\prod_{x \in \boldsymbol{D}} (G \times \mathcal{A}(Q_{x})) \right) \subseteq {\rm Hom}(\Omega, G).
\end{align*}
Given a $G$-representation $\rho: \Omega \rightarrow G$, we use the following notations
\begin{align*}
    a_i = \rho(\alpha_i), \ b_i = \rho(\beta_i), \ \rho(\gamma_x) = h_x, \ \rho(\gamma_{x,d}) =  S_{x,d}, \ \rho(\gamma_{0x}) = c_x.
\end{align*}
Then ${\rm Hom}_{\mathbb{S}}(\Pi,G)$ includes all data
\begin{align*}
    ((a_i,b_i)_{1\leq i\leq g}, (c_{x},h_{x},S_{x,d})_{x \in \boldsymbol{D}, d \in \mathbb{A}_x})
    \in {\rm Hom}_{\mathbb{S}}(\Omega,G)
\end{align*}
such that
\begin{align*}
    \left( \prod_{i =1}^g [a_i,b_i] \right) \cdot \left( \prod_{x \in \boldsymbol{D}} (c_{x}^{-1}h_{x} (\prod_{d \in \mathbb{A}_x} S_{x,d}) c_{x}) \right) =\mathrm{id}.
\end{align*}

Recall that $H_x$ is the stabilizer of $Q_x$ for $x \in \boldsymbol{D}$. We define
\begin{align*}
\boldsymbol{H}:=\prod_{x \in \boldsymbol{D}} H_x, \ \ \boldsymbol{H}(\partial):=\prod_{x \in \boldsymbol{D}} H(\partial_x).
\end{align*}
There is a $(G \times \boldsymbol{H})$-action on ${\rm Hom}_{\mathbb{S}}(\Pi,G)$ given as follows:
\begin{align*}
(g,(k_{x}&)_{x \in \boldsymbol{D}})\cdot((a_i,b_i)_{1\leq i\leq g}, (c_{x},h_{x},S_{x,d})_{x \in \boldsymbol{D}, d \in \mathbb{A}_{x}}):=\\
&\hspace{5em}((g a_i g^{-1},g b_i g^{-1})_{1\leq i\leq g}, (k_{x} c_{x} g^{-1},k_{x} h_{x} k^{-1}_{x},k_{x}S_{x,d}k_{x}^{-1})_{x \in \boldsymbol{D},d \in \mathbb{A}_{x}}).
\end{align*}
Under this action, ${\rm Hom}_{\mathbb{S}}(\Pi,G)$ becomes a (twsited) quasi-Hamiltonian $(G \times \boldsymbol{H})$-space with moment map
\begin{align*}
    \mu: {\rm Hom}_{\mathbb{S}}(\Pi,G)&\to G\times \boldsymbol{H}(\partial),\\
    \rho&\mapsto(\prod_{x \in \boldsymbol{D}} \Big(c_{x}^{-1}h_{x}(\prod_{d \in \mathbb{A}_x} S_{x,d}) c_{x}\Big), (h_x^{-1})_{x\in\boldsymbol{D}}).
\end{align*}
As a result, the quotient $\mathcal{M}_{\rm B}(X_{\boldsymbol{D}},G,\boldsymbol{Q}):={\rm Hom}_{\mathbb{S}}(\Pi,G)/\!\!/(G \times \boldsymbol{H})$, which is called \emph{wild character variety}, exhibits a structure of an algebraic Poisson variety with symplectic leaves \cite{Boa14, BY15}. Two Stokes $G$-representations with irregular type $\boldsymbol{Q}$ are isomorphic if they are in the same $(G \times \boldsymbol{H})$-orbit.

\begin{thm}[Theorem A.3 in \cite{Boa14}]\label{thm_boa14A.3}
There is a one-to-one correspondence between $(G \times \boldsymbol{H})$-orbits in ${\rm Hom}_{\mathbb{S}}(\Pi,G)$ and isomorphism classes of Stokes $G$-local systems with irregular type $\boldsymbol{Q}$ on $X_{\boldsymbol{D}}$. Thus, the category of Stokes $G$-local systems with irregular type $\boldsymbol{Q}$ on $X_{\boldsymbol{D}}$ is equivalent to the category of Stokes $G$-representations with irregular type $\boldsymbol{Q}$ on $X_{\boldsymbol{D}}$.
\end{thm}

\subsection{Stokes G-representations}\label{subsect_Sto_G_rep}
In the previous subsection, we follow Boalch's idea to construct the fundamental groupoid $\Pi$ of $X_{\boldsymbol{Q}}$ with respect to a collection of base points $\boldsymbol{b} = \{b_0, b_x \, x \in \boldsymbol{D}\}$, and then in the definition of the fundamental groupoid $\Pi$, it has a path (a generator) $\gamma_{0x}$ connecting $b_0$ and $b_x$ for each $x \in \boldsymbol{D}$. In the following, we will define a fundamental group of $X_{\boldsymbol{Q}}$ with respect to a single base point and give an equivalent description of the space of Stokes $G$-representations.

We define a free group $\Omega'$ with generators
\begin{enumerate}
\item $\alpha'_1,\beta'_1,\dots,\alpha'_g,\beta'_g$;
\item $\gamma'_x$ for each $x \in \boldsymbol{D}$;
\item $\gamma'_{x,d}$ for each $d \in \mathbb{A}_x$.
\end{enumerate}
Adding a relation
\begin{equation}\label{eq_rela_Pi'}\tag{$\ast'$}
    \left( \prod_{i=1}^g [\alpha'_i, \beta'_i] \right) \cdot \left( \prod_{x \in \boldsymbol{D}} \mu'_x \right) = {\rm id},
\end{equation}
where
\begin{align*}
    \mu'_x = \gamma'_x \cdot \left(  \prod_{d \in \mathbb{A}_x} \gamma'_{x,d}  \right),
\end{align*}
we obtain a group $\Pi'$. There is a morphism $\Omega' \rightarrow \Omega$ (resp. $\Pi' \rightarrow \Pi$)
\begin{align*}
    \alpha'_i \mapsto \alpha_i, \ \beta'_i \mapsto \beta_i, \ \gamma'_x \mapsto \gamma_{0x}^{-1} \gamma_x \gamma_{0x}, \ \gamma'_{x,d} \mapsto \gamma_{0x}^{-1} \gamma_{x,d} \gamma_{0x},
\end{align*}
which induces one ${\rm Hom}(\Omega,G) \rightarrow {\rm Hom}(\Omega',G)$ (resp. ${\rm Hom}(\Pi,G) \rightarrow {\rm Hom}(\Pi',G)$). Therefore, the group $\Pi'$ can be regarded as the fundamental group of $X_{\boldsymbol{Q}}$ with respect to a given base point $b_0$. Given a $G$-representation $\rho': \Omega' \rightarrow G$, we introduce the following notations
\begin{align*}
    a'_i = \rho'(\alpha'_i), \ b'_i = \rho'(\beta'_i), \ h'_x = \rho'(\gamma'_x), \ S'_{x,d} = \rho'(\gamma'_{x,d}).
\end{align*}

Consider the group $\prod_{x \in \boldsymbol{D}} G_x$, where $G_x := G$. We define an action
\begin{align*}
    (\prod_{x \in \boldsymbol{D}} G_x) \times {\rm Hom}(\Omega',G) \rightarrow {\rm Hom}(\Omega',G)
\end{align*}
via
\begin{align*}
    (g_x)_{x \in \boldsymbol{D}} \cdot ((a'_i,b'_i)_{1 \leq i \leq g} , (h'_x, S'_{x,d})_{x \in \boldsymbol{D},d \in \mathbb{A}_x}  ) := ((a'_i,b'_i)_{1 \leq i \leq g} , (g^{-1}_x h'_x g_x, g^{-1}_x S'_{x,d} g_x )_{x \in \boldsymbol{D}, d \in \mathbb{A}_x}  ).
\end{align*}
Consider the fiber product
\begin{align*}
    \left( (\prod_{x \in \boldsymbol{D}} G_x) \times {\rm Hom}(\Omega',G) \right) \times_{ {\rm Hom}(\Omega',G) } \left( (G \times G)^g \times \prod_{x \in \boldsymbol{D}} \mathcal{A}(Q_x)  \right),
\end{align*}
where $\left( (G \times G)^g \times \prod_{x \in \boldsymbol{D}} \mathcal{A}(Q_x)  \right) \hookrightarrow {\rm Hom}(\Omega',G)$ is the natural inclusion. The fiber product is a closed subvariety of $(\prod_{x \in \boldsymbol{D}} G_x) \times {\rm Hom}(\Omega',G)$, and it includes all points
\begin{align*}
    ((g_x)_{x \in \boldsymbol{D}} , ((a'_i,b'_i)_{1 \leq i \leq g} , (h'_x, S'_{x,d})_{x \in \boldsymbol{D},d \in \mathbb{A}_x}  )) \in (\prod_{x \in \boldsymbol{D}} G_x) \times {\rm Hom}(\Omega',G)
\end{align*}
such that
\begin{align*}
    g^{-1}_x h'_x g_x \in H(\partial_x), \ g^{-1}_x S'_{x,d} g_x \in \mathbb{S}{\rm to}_d
\end{align*}
for each $x \in \boldsymbol{D}$ and $d \in \mathbb{A}_x$. Then, we define
\begin{align*}
    {\rm Hom}_{\mathbb{S}} (\Omega', G): = \left( ( (\prod_{x \in \boldsymbol{D}} G_x) \times {\rm Hom}(\Omega',G) ) \times_{ {\rm Hom}(\Omega',G) } ( (G \times G)^g \times \prod_{x \in \boldsymbol{D}} \mathcal{A}(Q_x) ) \right) \bigg|_{ {\rm Hom}(\Omega',G) }.
\end{align*}
Clearly, ${\rm Hom}_{\mathbb{S}} (\Omega', G)$ is a locally closed subset and includes all points
\begin{align*}
    ((a'_i,b'_i)_{1 \leq i \leq g} , (h'_x, S'_{x,d})_{x \in \boldsymbol{D},d \in \mathbb{A}_x}  ) \in {\rm Hom}(\Omega',G)
\end{align*}
satisfying the condition that for each $x \in \boldsymbol{D}$, there exists $g_x \in G$ such that
\begin{align*}
    g^{-1}_x h'_x g_x \in H(\partial_x), \ g^{-1}_x S'_{x,d} g_x \in \mathbb{S}{\rm to}_d.
\end{align*}
Adding the relation \eqref{eq_rela_Pi'}, we obtain a closed subvariety ${\rm Hom}_{\mathbb{S}}(\Pi',G) \hookrightarrow {\rm Hom}_{\mathbb{S}}(\Omega',G)$. Furthermore, the natural $G$-action on ${\rm Hom}(\Omega', G)$ given by conjugation
\begin{align*}
    g \cdot ((a'_i,b'_i)_{1 \leq i \leq g} , (h'_x, S'_{x,d})_{x \in \boldsymbol{D},d \in \mathbb{A}_x}  ) := ((g a'_i g^{-1},g b'_i g^{-1})_{1\leq i\leq g}, (g h'_{x} g^{-1},g S'_{x,d} g^{-1})_{x \in \boldsymbol{D},d \in \mathbb{A}_{x}})
\end{align*}
induces a $G$-action on ${\rm Hom}_{\mathbb{S}}(\Pi',G)$.

\begin{prop}\label{prop_reps}
There is a one-to-one correspondence between $(G \times \boldsymbol{H})$-orbits in ${\rm Hom}_{\mathbb{S}}(\Pi,G)$ and $G$-orbits in ${\rm Hom}_{\mathbb{S}}(\Pi',G)$.
\end{prop}

\begin{proof}
There is a natural morphism
\begin{align*}
    {\rm Hom}_{\mathbb{S}}(\Pi,G) \rightarrow {\rm Hom}_{\mathbb{S}}(\Pi',G)
\end{align*}
given by
\begin{align*}
    ((a_i,b_i)_{1\leq i\leq g}, (c_{x},h_{x},S_{x,d})_{x \in \boldsymbol{D},d \in \mathbb{A}_{x}}) \rightarrow ( (a_i,b_i)_{1 \leq i \leq g}, (c_x^{-1} h_x c_x, c_x^{-1} S_{x,d} c_x)_{x \in \boldsymbol{D},d \in \mathbb{A}_{x}} ).
\end{align*}
In other words,
\begin{align*}
    a'_i = a_i, \quad b'_i = b_i, \quad h'_x = c^{-1}_x h_x c_x, \quad S'_{x,d} = c^{-1}_x S_{x,d} c_x.
\end{align*}
Given an arbitrary element $(g,(k_x)_{x \in \boldsymbol{D}}) \in G \times \boldsymbol{H}$, we have
\begin{align*}
(g,(k_{x}&)_{x \in \boldsymbol{D}})\cdot((a_i,b_i)_{1\leq i\leq g}, (c_{x},h_{x},S_{x,d})_{x \in \boldsymbol{D}, d \in \mathbb{A}_{x}})=\\
&\hspace{5em}((g a_i g^{-1},g b_i g^{-1})_{1\leq i\leq g}, (k_{x} c_{x} g^{-1},k_{x} h_{x} k^{-1}_{x},k_{x}S_{x,d}k_{x}^{-1})_{x \in \boldsymbol{D},d \in \mathbb{A}_{x}}).
\end{align*}
Moreover,
\begin{align*}
g \cdot  ( (a'_i,b'_i)_{1 \leq i \leq g}, (h'_x, S'_{x,d} )_{x \in \boldsymbol{D}, d \in \mathbb{A}_{x}} ) = ( (g a'_i g^{-1}, g b'_i g^{-1})_{1 \leq i \leq g}, (g h'_x g^{-1},
g S'_{x,d} g^{-1} )_{x \in \boldsymbol{D}, d \in \mathbb{A}_{x}}  ).
\end{align*}
Clearly, the image of
\begin{align*}
    (g,(k_{x})_{x \in \boldsymbol{D}})\cdot((a_i,b_i)_{1\leq i\leq g}, (c_{x},h_{x},S_{x,d})_{x \in \boldsymbol{D}, d \in \mathbb{A}_{x}}) \in {\rm Hom}_{\mathbb{S}}(\Pi, G)
\end{align*}
is
\begin{align*}
    g \cdot  ( (a'_i,b'_i)_{1 \leq i \leq g}, (h'_x, S'_{x,d} )_{x \in \boldsymbol{D}, d \in \mathbb{A}_{x}} ) \in {\rm Hom}_{\mathbb{S}}(\Pi',G).
\end{align*}
The proposition follows directly.
\end{proof}

We have the following corollary as a direct result of Theorem \ref{thm_boa14A.3} and Proposition \ref{prop_reps}.

\begin{cor}\label{cor_Stokes_rep_and_loc}
There is a one-to-one correspondence between $G$-orbits in ${\rm Hom}_{\mathbb{S}}(\Pi',G)$ and isomorphism classes of Stokes $G$-local systems with irregular type $\boldsymbol{Q}$ on $X_{\boldsymbol{D}}$.
\end{cor}

\begin{termi*}
From now on, a representation in ${\rm Hom}_{\mathbb{S}}(\Pi',G)$ will be called a \emph{Stokes $G$-representation on $X_{\boldsymbol{D}}$}.
\end{termi*}

\section{Moduli Space of Filtered Stokes G-Local Systems}\label{sect_moduli}

In this section, we construct the moduli space of filtered Stokes $G$-local systems with irregular type $\boldsymbol{Q}$ on $X_{\boldsymbol{D}}$. By Proposition \ref{prop_reps} and Corollary \ref{cor_Stokes_rep_and_loc}, it is equivalent to construct the moduli space for filtered Stokes $G$-representations with irregular type $\boldsymbol{Q}$. In \S\ref{subsect_stab_fil_G}, we first give the stability condition for filtered Stokes $G$-representation (Definition \ref{defn_stab_cond_G_Stokes}) based on Ramanathan's approach \cite{Ram75,Ram96a}. In \S\ref{subsect_cons_pi''}, we give a third construction of the space of Stokes $G$-representations, which will be used in the construction of the moduli space. In \S\ref{subsect_mod_space}, we prove that the stability condition of filtered Stokes $G$-representations is equivalent to a stability condition in the sense of GIT (Proposition \ref{prop_stab_equiv}), and then we follow King's approach \cite[\S 2]{King94} to construct the moduli space (Theorem \ref{thm_Sto_moduli}). In this section, since we always fix an irregular type $\boldsymbol{Q}$, if there is no ambiguity, we use the terminology filtered Stokes $G$-representations (or filtered Stokes $G$-local systems) without mentioning the irregular type.

\subsection{Stability Condition of Filtered Stokes G-local Systems}\label{subsect_stab_fil_G}

Recall that $G$ is a connected complex reductive group with a maximal torus $T$. Denote by $\mathcal{R}$ the set of roots. There is a natural pairing of cocharacters and characters
\begin{align*}
    \langle \cdot , \cdot \rangle : {\rm Hom}(\mathbb{G}_m ,T) \times {\rm Hom}(T, \mathbb{G}_m) \rightarrow \mathbb{Z}.
\end{align*}
This pairing can be extended to cocharacters and characters with rational coefficients, and a rational cocharacter in this paper is also called a \emph{weight}.
Now we fix a Borel subgroup $B$, which includes $T$. Let $P$ be a parabolic subgroup. Given a character $\chi$ of $P$ and a cocharacter $\mu$ of $T$, we define
\begin{align*}
    \langle \mu , \chi \rangle := \langle g^{-1} \mu g , \chi \rangle,
\end{align*}
where $g$ satisfies $B \subseteq g P g^{-1}$ and $g^{-1} \mu g$ is a cocharacter of $P$. Furthermore, the definition $\langle \mu , \chi \rangle$ does not depend on the choice of such element $g$.

Let $\{e_i\}$ (resp. $\{e^*_i\}$) be a basis of ${\rm Hom}(\mathbb{G}_m,T) \otimes_{\mathbb{Z}} \mathbb{Q}$ (resp. ${\rm Hom}(T,\mathbb{G}_m) \otimes_{\mathbb{Z}} \mathbb{Q}$) such that $\langle e_i, e^*_j \rangle = \delta_{ij}$. Suppose that $G$ is semisimple for convenience, and then, $\{e_i\}$ is regarded as a collection of simple coroots, while $\{e^*_i\}$ is regarded as the set of the corresponding fundamental weights. Given a cocharacter $\theta$, a character $\chi_{\theta}$ is uniquely determined by the conditions
\begin{align*}
    \langle e_i, \chi_\theta \rangle  = \langle \theta , e^*_i \rangle
\end{align*}
for each $i$, and similarly, a cocharacter $\theta_{\chi}$ is determined by a given character $\chi$ by the conditions
\begin{align*}
    \langle \theta_\chi, e^*_i \rangle  = \langle e_i , \chi \rangle
\end{align*}
for each $i$. Clearly, we have
\begin{align*}
    \langle \theta, \chi \rangle = \langle \theta_\chi , \chi_\theta \rangle.
\end{align*}

Now let $\theta$ be a weight. It determines a parabolic subgroup
\begin{align*}
    P_\theta: = \{g \in G \, | \, \text{ the limit } \lim_{t \rightarrow 0} \theta(t) g \theta(t)^{-1} \text{ exists }\}
\end{align*}
with Levi subgroup $L_\theta$. Here is another interpretation of $P_\theta$. Define
\begin{align*}
    \mathcal{R}_\theta := \{ \alpha \in \mathcal{R} \, | \, \langle \theta , \alpha \rangle \geq 0 \}.
\end{align*}
Then $P_\theta = \langle T, \, U_\alpha, \, \alpha \in \mathcal{R}_\theta \rangle$, i.e. $P_\theta$ is generated by $T$ and $U_\alpha$ for $\alpha \in \mathcal{R}_\theta$. On the other hand, given a parabolic subgroup $P \subseteq G$, denote by $\mathcal{R}_P$ the set of roots of $P$. Clearly, $\mathcal{R}_{P_\theta} = \mathcal{R}_\theta$. Now we consider a special type of characters, which is called \emph{dominant characters} and introduced by \cite[\S 2]{Ram75}.
\begin{defn}
Given a parabolic subgroup $P$, a character $\chi$ of $P$ is called \emph{dominant} (resp. \emph{anti-dominant}) if it is a positive (resp. negative) linear combination of fundamental weights given by roots in $\mathcal{R}_P$.
\end{defn}
In \cite{HS23}, the authors proved the following lemma and a similar argument is also given in \cite[Lemma 2.2]{MiR18}.
\begin{lem}[Lemma 4.6 in \cite{HS23}]\label{lem_char_cochar}
Given a weight $\theta$, the character $\chi_\theta$ is a dominant character of $P_\theta$. On the other hand, given a character $\chi$, if it is a dominant character of some parabolic subgroup $P$, then $P_{\theta_\chi} \supseteq P$.
\end{lem}

\begin{defn}
Let $\boldsymbol\theta = \{\theta_x, x \in \boldsymbol{D}\}$ be a collection of weights. A \emph{$\boldsymbol\theta$-filtered Stokes $G$-representation} is a Stokes $G$-representation $\rho'$ such that the formal monodromy $h'_x = \rho'(\gamma'_x)$ is conjugate to an element in $P_{\theta_x}$ for every $x \in \boldsymbol{D}$. The corresponding Stokes $G$-local system is called a \emph{$\boldsymbol\theta$-filtered Stokes $G$-local system}.
\end{defn}

It is well-known that a connected complex reductive group $G$ is covered by its Borel subgroups. Clearly, the statement also holds for parabolic subgroups. Then, fixing an arbitrary parabolic subgroup $P$, any $g \in G$ is conjugate to an element in $P$. Therefore, the space of $\boldsymbol\theta$-filtered Stokes $G$-representations can also be regarded as ${\rm Hom}_{\mathbb{S}}(\Pi',G)$.

Given a Stokes $G$-representation $\rho' : \Pi' \rightarrow G$, a parabolic subgroup $P$ is \emph{compatible} with $\rho'$, if there is a lifting
\begin{center}
\begin{tikzcd}
& & P \arrow[d] \\
\Pi' \arrow[urr, dotted] \arrow[rr,"\rho'"] & & G
\end{tikzcd}
\end{center}
In other words, the representation $\rho'$ is well-defined when restricted to $P$. Let $L$ be the Levi subgroup of $P$. If $P$ is compatible with $\rho'$, then $\rho'$ is also well-defined by restricting to $L$. Under the morphism $P \twoheadrightarrow L \rightarrow G$, we obtain a $G$-representation and denote it by $\rho'_L$.

Given a $\boldsymbol\theta$-filtered Stokes $G$-representation $\rho'$, there exists $g_x \in G$ such that $g_x \rho'(\gamma'_x) g^{-1}_x \in P_{-\theta_x}$ for each $x \in \boldsymbol{D}$. Suppose that the parabolic subgroup $P$ is compatible with $\rho'$, and then $B_{-\theta_x} \subseteq g_x P g^{-1}_x$, where $B_{-\theta_x} \subseteq P_{-\theta_x}$ is the Borel subgroup. Let $\chi$ be a character of $P$ and the natural pairing is given as
\begin{align*}
    \langle \theta_x , \chi \rangle = \langle g^{-1}_x \theta_x g_x , \chi \rangle = \langle \theta_x , g_x \chi g_x^{-1} \rangle.
\end{align*}
We define the degree of a $\boldsymbol\theta$-filtered Stokes $G$-representation $\rho'$ as
\begin{align*}
   \deg^{\rm loc} \rho'(P,\chi) := \langle \boldsymbol\theta, \chi \rangle = \sum_{x \in \boldsymbol{D}} \langle \theta_x, \chi \rangle.
\end{align*}
Furthermore, a parabolic subgroup $P$ is \emph{admissible} with $\rho'$ if $P$ is compatible with $\rho'$ and for any character $\chi: P \rightarrow \mathbb{G}_m$ trivial on the center, we have $\deg^{\rm loc} \rho' (P, \chi)=0$.

We follow Ramanathan's stability condition on principal bundles \cite{Ram75,Ram96a} to give the definition of stability condition on filtered Stokes $G$-representations (also for filtered Stokes $G$-local systems), which is called the \emph{$R$-stability condition}.

\begin{defn}\label{defn_stab_cond_G_Stokes}
A $\boldsymbol\theta$-filtered Stokes $G$-representation $\rho'$ is \emph{$R$-semistable} (resp. \emph{$R$-stable}), if for
\begin{itemize}
\item any proper parabolic subgroup $P \subseteq G$ compatible with $\rho'$,
\item any nontrivial anti-dominant character $\chi: P \rightarrow \mathbb{G}_m$, which is trivial on the center of $P$,
\end{itemize}
we have
\begin{align*}
\deg^{\rm loc} \rho'(P,\chi) \geq 0  \quad  (\text{resp.} > 0).
\end{align*}
Moreover, two $R$-semistable $\boldsymbol\theta$-filtered Stokes $G$-representations $\rho'_1$ and $\rho'_2$ are \emph{$S$-equivalent} if there exist parabolic subgroups $P_1$ and $P_2$ (with Levi subgroups $L_1$ and $L_2$) admissible with $\rho'_1$ and $\rho'_2$ respectively such that the corresponding Stokes $G$-representations $(\rho'_1)_{L_1}$ and $(\rho'_2)_{L_2}$ are conjugate under the action of $G$.
\end{defn}

\begin{defn}
A $\boldsymbol\theta$-filtered Stokes $G$-representation $\rho'$ is of \emph{degree zero}, if for any character $\chi$ of $G$, we have $\deg^{\rm loc} \rho'(P,\chi) = 0$. Note that when $G$ is semisimple, this condition is always satisfied.
\end{defn}

\subsection{An Equivalent Construction}\label{subsect_cons_pi''}
In this subsection, we give a third construction of the space of Stokes $G$-representations. We define a free group $\Omega''$ generated by the following elements
\begin{enumerate}
\item $\alpha''_1,\beta''_1,\dots,\alpha''_g,\beta''_g$;
\item $\iota''_x, \gamma''_x$ for each $x \in \boldsymbol{D}$;
\item $\gamma''_{x,d}$ for each $d \in \mathbb{A}_x$.
\end{enumerate}
Given a relation
\begin{equation}\label{eq_rela_Pi''}\tag{$\ast''$}
    \left( \prod_{i=1}^g [\alpha''_i, \beta''_i] \right) \cdot \left( \prod_{x \in \boldsymbol{D}} \mu''_x \right) = {\rm id},
\end{equation}
where
\begin{align*}
    \mu''_x = \iota''_x \cdot \gamma''_x \cdot \left(  \prod_{d \in \mathbb{A}_x} \gamma''_{x,d}  \right),
\end{align*}
we obtain a group $\Pi''$. The natural surjection $\Omega'' \rightarrow \Pi''$ induces a closed embedding ${\rm Hom}(\Pi'',G) \hookrightarrow {\rm Hom}(\Omega'',G)$. Moreover, given a $G$-representation $\rho'': \Omega'' \rightarrow G$ (or $\rho'': \Pi'' \rightarrow G$), we introduce the following notations
\begin{align*}
    a''_i = \rho''(\alpha''_i), \ b''_i = \rho''(\beta''_i), \ \rho''(\iota''_x) = l''_x, \ \rho''(\gamma''_x) = h''_x, \ \rho''(\gamma''_{x,d}) = S''_{x,d}.
\end{align*}
We define a morphism $\Omega' \rightarrow \Omega''$ (resp. $\Pi' \rightarrow \Pi''$)
\begin{align*}
    \alpha' \mapsto \alpha'' , \ \beta' \mapsto \beta'', \ \gamma'_x \mapsto \iota''_x \gamma''_x , \ \gamma'_{x,d} \mapsto \gamma''_{x,d},
\end{align*}
which induces a morphism ${\rm Hom}(\Omega'',G) \rightarrow {\rm Hom}(\Omega',G)$ (resp. ${\rm Hom}(\Pi'',G) \rightarrow {\rm Hom}(\Pi',G)$). Taking the fiber product
\begin{center}
\begin{tikzcd}
{\rm Hom}_{\mathbb{S}}(\Omega'',G) \arrow[rr, dotted] \arrow[d, dotted] & & {\rm Hom}_{\mathbb{S}}(\Omega',G) \arrow[d, hook] \\
{\rm Hom}(\Omega'',G) \arrow[rr] & & {\rm Hom}(\Omega',G) \ ,
\end{tikzcd}
\end{center}
we obtain a quasi-projective variety ${\rm Hom}_{\mathbb{S}}(\Omega'',G)$, which includes all points
\begin{align*}
    ((a''_i,b''_i)_{1 \leq i \leq g} , (l''_x, h''_x, S''_{x,d})_{x \in \boldsymbol{D},d \in \mathbb{A}_x}  ) \in {\rm Hom}(\Omega'',G)
\end{align*}
satisfying the condition that for each $x \in \boldsymbol{D}$, there exists $g_x \in G$ such that
\begin{align*}
    g^{-1}_x l''_x h''_x g_x \in H(\partial_x), \ g^{-1}_x S''_{x,d} g_x \in \mathbb{S}{\rm to}_d.
\end{align*}
Moreover, we define a $(\prod_{x \in \boldsymbol{D}} G_x)$-action on ${\rm Hom}_{\mathbb{S}}(\Omega'',G)$ via
\begin{align*}
    (g_x)_{x \in \boldsymbol{D}} \cdot ((a''_i,b''_i)_{1 \leq i \leq g} , (l''_x, h''_x, S''_{x,d})_{x \in \boldsymbol{D},d \in \mathbb{A}_x}  ) := ((a''_i,b''_i)_{1 \leq i \leq g} , (g^{-1}_x l''_x, h''_x g_x, S''_{x,d})_{x \in \boldsymbol{D},d \in \mathbb{A}_x}  ),
\end{align*}
where $G_x=G$ for each $x \in \boldsymbol{D}$.

Now given a collection of weights $\boldsymbol\theta = \{\theta_x, x \in \boldsymbol{D}\}$, denote by $\boldsymbol{P} = \{P_{-\theta_x}, x \in \boldsymbol{D}\}$ the collection of parabolic subgroups. We define a closed subvariety ${\rm Hom}_{\mathbb{S}}(\Omega'', \boldsymbol{P}) \subseteq {\rm Hom}_{\mathbb{S}}(\Omega'', G)$, of which points
\begin{align*}
     ((a''_i,b''_i)_{1 \leq i \leq g} , (l''_x, h''_x, S''_{x,d})_{x \in \boldsymbol{D},d \in \mathbb{A}_x}  )
\end{align*}
satisfy the condition that
\begin{align*}
    l''_x \in L_{\theta_x}, \ h''_x \in P_{-\theta_x}
\end{align*}
for each $x \in \boldsymbol{D}$. We take the fiber product
\begin{center}
\begin{tikzcd}
\widetilde{{\rm Hom}}_{\mathbb{S}}(\Omega'',\boldsymbol{P}) \arrow[rr, dotted] \arrow[d, dotted] & & {\rm Hom}_{\mathbb{S}}(\Omega'',\boldsymbol{P}) \arrow[d, hook] \\
(\prod_{x \in \boldsymbol{D}} G_x) \times {\rm Hom}_{\mathbb{S}}(\Omega'',G) \arrow[rr] & & {\rm Hom}_{\mathbb{S}}(\Omega'',G) \ .
\end{tikzcd}
\end{center}
Then we restrict it to ${\rm Hom}(\Omega'',G)$ and define
\begin{align*}
    {\rm Hom}_{\mathbb{S}}(\Omega'',[\boldsymbol{P}]) := \widetilde{{\rm Hom}}_{\mathbb{S}}(\Omega'',\boldsymbol{P}) |_{ {\rm Hom}_{\mathbb{S}}(\Omega'',G) }.
\end{align*}
The variety ${\rm Hom}_{\mathbb{S}}(\Omega'',[\boldsymbol{P}])$ includes all points
\begin{align*}
    ((a''_i,b''_i)_{1 \leq i \leq g} , (l''_x, h''_x, S''_{x,d})_{x \in \boldsymbol{D},d \in \mathbb{A}_x}  ) \in {\rm Hom}_{\mathbb{S}}(\Omega'',G)
\end{align*}
such that for each $x \in \boldsymbol{D}$, there exists $g_x \in G$ such that $g^{-1}_x l''_x \in L_{\theta_x}$ and $h''_x g_x \in P_{-\theta_x}$. Then we obtain a closed subvariety ${\rm Hom}_{\mathbb{S}}(\Pi'', [\boldsymbol{P}]) \subseteq {\rm Hom}_{\mathbb{S}}(\Omega'', [\boldsymbol{P}])$ by adding the relation \eqref{eq_rela_Pi''}. Since the collection $\boldsymbol{P}$ of parabolic subgroups is determined by $\boldsymbol\theta$, we would like to use the notation
\begin{align*}
    {\rm Hom}_{\mathbb{S}}(\Pi'', \boldsymbol\theta) : =  {\rm Hom}_{\mathbb{S}}(\Pi'', [\boldsymbol{P}]).
\end{align*}
Define $\boldsymbol{L} = \prod_{x \in \boldsymbol{D}} L_{\theta_x}$. There is a $\boldsymbol{L}$-action on ${\rm Hom}_{\mathbb{S}}(\Pi'',\boldsymbol\theta)$ defined as follows
\begin{align*}
    (l_x)_{x \in \boldsymbol{D}} \, \cdot \, & ( (a''_i,b''_i)_{1 \leq i \leq g}, (l''_x, h''_x, S''_{x,d} )_{x \in \boldsymbol{D}, d \in \mathbb{A}_{x}} ) \\
    := & ( (a''_i,b''_i)_{1 \leq i \leq g}, (l''_x l_x^{-1}, l_x h''_x, S''_{x,d} )_{x \in \boldsymbol{D}, d \in \mathbb{A}_{x}} )
\end{align*}
Then we define a $(G \times \boldsymbol{L})$-action on ${\rm Hom}_{\mathbb{S}}(\Pi'', \boldsymbol{\theta})$
\begin{align*}
    (g, (l_x)_{x \in \boldsymbol{D}}) \, \cdot \, & ( (a''_i,b''_i)_{1 \leq i \leq g}, (l''_x, h''_x, S''_{x,d} )_{x \in \boldsymbol{D},d \in \mathbb{A}_{x}} ) \\
    := & ( (g a''_i g^{-1}, g b''_i g^{-1})_{1 \leq i \leq g}, (g l''_x l_x^{-1}, l_x h''_x g^{-1}, g S''_{x,d} g^{-1})_{x \in \boldsymbol{D}, d \in \mathbb{A}_{x}} ).
\end{align*}

\begin{lem}\label{lem_corr_pi'_pi''}
There is a one-to-one correspondence between $(G \times \boldsymbol{L})$-orbits in ${\rm Hom}_{\mathbb{S}}(\Pi'', \boldsymbol{\theta})$ and $G$-orbits in ${\rm Hom}_{\mathbb{S}}(\Pi',G)$. Therefore, $(G \times \boldsymbol{L})$-orbits in ${\rm Hom}_{\mathbb{S}}(\Pi'',\boldsymbol\theta)$ are in one-to-one correspondence with isomorphism classes of $\boldsymbol\theta$-filtered Stokes $G$-representations, and thus isomorphism classes of $\boldsymbol\theta$-filtered Stokes $G$-local systems.
\end{lem}

\begin{proof}
The surjective morphism
\begin{align*}
    {\rm Hom}(\Pi'',G) \rightarrow {\rm Hom}(\Pi',G)
\end{align*}
induces the surjection
\begin{align*}
    {\rm Hom}_{\mathbb{S}}(\Pi'',\boldsymbol\theta) \rightarrow {\rm Hom}_{\mathbb{S}}(\Pi',G)
\end{align*}
given by
\begin{align*}
    ( (a''_i,b''_i)_{1 \leq i \leq g}, (l''_x, h''_x, S''_{x,d} )_{x \in \boldsymbol{D}, d \in \mathbb{A}_{x}} ) \rightarrow ( (a''_i,b''_i)_{1 \leq i \leq g}, (l''_x h''_x, S''_{x,d} )_{x \in \boldsymbol{D}, d \in \mathbb{A}_{x}} ).
\end{align*}
Given an arbitrary element $(g, (l_x)_{x \in \boldsymbol{D}}) \in G \times \boldsymbol{L}$, the image of
\begin{align*}
    (g, (l_x)_{x \in \boldsymbol{D}}) \, \cdot \, ( (a''_i,b''_i)_{1 \leq i \leq g}, (l''_x, h''_x, S''_{x,d} )_{x \in \boldsymbol{D}, d \in \mathbb{A}_{x}} ) \in {\rm Hom}_{\mathbb{S}}(\Pi'', \boldsymbol\theta)
\end{align*}
is exactly
\begin{align*}
    g \, \cdot \, (a''_i,b''_i)_{1 \leq i \leq g}, (l''_x h''_x, S''_{x,d} )_{x \in \boldsymbol{D}, d \in \mathbb{A}_{x}} ) \in {\rm Hom}_{\mathbb{S}} (\Pi',G).
\end{align*}
This finishes the proof of this lemma.
\end{proof}

\subsection{Moduli Space}\label{subsect_mod_space}

In this subsection, we follow King's approach \cite{King94} to construct the moduli space of filtered Stokes $G$-local systems. We fix a collection of weights $\boldsymbol\theta$ and we suppose that $d$ is the common denominator of $\theta_x$ for $x \in \boldsymbol{D}$, i.e. $d \theta_x$ is a cocharacter for every $x \in \boldsymbol{D}$. In the previous subsection, we construct a quasi-projective variety ${\rm Hom}_{\mathbb{S}}(\Pi'',\boldsymbol\theta)$ with a natural $(G \times \boldsymbol{L})$-action such that the $(G \times \boldsymbol{L})$-orbits are in one-to-one correspondence with isomorphism classes of $\boldsymbol\theta$-filtered Stokes $G$-representations by Lemma \ref{lem_corr_pi'_pi''}. We will introduce a particular character $\chi_{\boldsymbol\theta}: G \times \boldsymbol{L} \rightarrow \mathbb{G}_m$ such that a $\boldsymbol\theta$-filtered Stokes $G$-representation $\rho'$ is $R$-semistable if and only if the corresponding representation $\rho'' \in {\rm Hom}_{\mathbb{S}}(\Pi'',\boldsymbol\theta)$ is $\chi_{\boldsymbol\theta}$-semistable in the sense of GIT. Based on the equivalence of stability conditions, we use ${\rm Hom}_{\mathbb{S}}(\Pi'',\boldsymbol\theta)$ to construct the moduli space.

We define a character
\begin{align*}
    \chi_{\boldsymbol\theta}: G \times \boldsymbol{L} \rightarrow \mathbb{G}_m,
\end{align*}
as
\begin{align*}
    \chi_{\boldsymbol\theta} (g,(l_x)_{x \in \boldsymbol{D}}) = \chi_0(g) \cdot \prod_{x \in \boldsymbol{D}} \chi_{-d \theta_x}(l_x),
\end{align*}
where $\chi_0$ is the trivial character of $G$ and the character $\chi_{-d\theta_x}$ is determined by the weight $-d \theta_x$. Given a cocharacter $\lambda: \mathbb{G}_m \rightarrow G \times \boldsymbol{L}$, it is given by a cocharacter $\lambda_0$ of $G$ and a cocharacter $\lambda_x$ of $L_{\theta_x}$ for each $x \in \boldsymbol{D}$. Thus, the pairing $\langle \lambda, \chi_{\boldsymbol\theta} \rangle$ is given by
\begin{align*}
    \langle \lambda, \chi_{\boldsymbol\theta} \rangle = \langle \lambda_0, \chi_0 \rangle + \sum_{x \in \boldsymbol{D}}\langle \lambda_x , \chi_{-d\theta_x} \rangle = \sum_{x \in \boldsymbol{D}}\langle \lambda_x , \chi_{-d\theta_x} \rangle.
\end{align*}
With respect to the $(G \times \boldsymbol{L})$-action and character $\chi_{\boldsymbol\theta}$, King defined the GIT quotient ${\rm Hom}_{\mathbb{S}}(\Pi'',\boldsymbol\theta) / \! \! / (G \times \boldsymbol{L}, \chi_{\boldsymbol\theta})$, which parametrizes GIT equivalence classes of $\chi_{\boldsymbol\theta}$-semistable points in ${\rm Hom}_{\mathbb{S}}(\Pi'',\boldsymbol\theta)$. We refer the reader to \cite[\S 2]{King94} for more details about this construction. Applying \cite[Proposition 2.5, 2.6]{King94}, we have the following equivalent description of $\chi_{\boldsymbol\theta}$-semistable points.

\begin{lem}\label{lem_King_prop2.5}
Denote by $\Delta$ the kernel of the $(G \times \boldsymbol{L})$-action on ${\rm Hom}_{\mathbb{S}}(\Pi'',\boldsymbol\theta)$. A point $\rho'' \in {\rm Hom}_{\mathbb{S}}(\Pi'',\boldsymbol\theta)$ is $\chi_{\boldsymbol\theta}$-semistable if and only if $\chi_{\boldsymbol\theta}(\Delta) = \{1\}$ and any cocharacter $\lambda$ of $G \times \boldsymbol{L}$, for which the limit $\lim_{t \rightarrow 0} \lambda(t) \cdot \rho''$ exists, satisfies $\langle \lambda , \chi_{\boldsymbol\theta} \rangle \geq 0$. It is $\chi_{\boldsymbol\theta}$-stable if and only if any cocharacter $\lambda$, for which $\lim_{t \rightarrow 0} \lambda(t) \cdot \rho''$ exists and $\langle \lambda, \chi_{\boldsymbol\theta} \rangle = 0$, is in $\Delta$. Moreover, two $\chi_{\boldsymbol\theta}$-semistable points $\rho''_1$ and $\rho''_2$ are GIT equivalent if and only if there are cocharacters $\lambda_1$ and $\lambda_2$ such that $\langle \lambda_1, \chi_{\boldsymbol\theta} \rangle = \langle \lambda_2, \chi_{\boldsymbol\theta} \rangle = 0$ and the limits $\lim_{t \rightarrow 0} \lambda_1(t) \cdot \rho''_1$ and $\lim_{t \rightarrow 0} \lambda_2(t) \cdot \rho''_2$ are in the same $(G \times \boldsymbol{L})$-orbit.
\end{lem}

\begin{rem}\label{rem_cochar}
We give a precise description of the cocharacter $\lambda$ such that $\lim\limits_{t \rightarrow 0} \lambda(t) \cdot \rho''$ exists, and we regard $\rho''$ as a tuple
\begin{align*}
( (a''_i,b''_i)_{1 \leq i \leq g}, (l''_x, h''_x, S''_{x,d} )_{x \in \boldsymbol{D}, d \in \mathbb{A}_{x}} ).
\end{align*}
Given a cocharacter $\lambda: \mathbb{G}_m \rightarrow G \times \boldsymbol{L}$, it is uniquely determined by a cocharacter $\lambda_0$ of $G$ and a cocharacter $\lambda_x$ of $L_{\theta_x}$ for each $x \in \boldsymbol{D}$. Suppose that the limit $\lim\limits_{t \rightarrow 0} \lambda(t) \cdot \rho''$ exists. Then the existence of the limits
\begin{align*}
    \lim\limits_{t \rightarrow 0} \lambda_0(t) a''_i \lambda^{-1}_0(t), \ \lim\limits_{t \rightarrow 0} \lambda_0(t) b''_i \lambda^{-1}_0(t), \ \lim\limits_{t \rightarrow 0} \lambda_0(t) S''_{x,d} \lambda^{-1}_0(t)
\end{align*}
implies that $a''_i, b''_i, S''_{x,d} \in P_{\lambda_0}$, and the existence of the limits
\begin{align*}
    \lim\limits_{t \rightarrow 0} \lambda_0(t) l''_x \lambda^{-1}_x(t), \ \lim\limits_{t \rightarrow 0} \lambda_x(t) h''_x \lambda^{-1}_0(t)
\end{align*}
implies that $l''_x h''_x \in P_{\lambda_0}$. Therefore, the corresponding representation $\rho' \in {\rm Hom}_{\mathbb{S}}(\Pi',G)$ of $\rho''$ (under the morphism $\Pi' \rightarrow \Pi''$) is compatible with $P_{\lambda_0}$. Moreover, we claim that for each $x \in \boldsymbol{D}$, there exists $g_x \in G$ such that $\lambda_x(t) = g^{-1}_x \lambda_0(t) g_x$, $g^{-1}_x l''_x  \in L_{\theta_x}$ and $h''_x g_x \in P_{-\theta_x}$. Therefore, we have
\begin{equation}\tag{$\bullet$}\label{eq_lambda_chi}
    \langle \lambda, \chi_{\boldsymbol\theta} \rangle = \sum_{x \in \boldsymbol{D}} \langle \lambda_x, \chi_{-d \theta_x} \rangle = -d \sum_{x \in \boldsymbol{D}} \langle \theta_x, \chi_{\lambda_x} \rangle = -d \sum_{x \in \boldsymbol{D}} \langle \theta_x, \chi_{\lambda_0} \rangle = -d \langle \boldsymbol\theta , \chi_{\lambda_0} \rangle.
\end{equation}

\space{\hfill}

Here is a brief explanation for the claim. By construction, for each $x \in \boldsymbol{D}$, there exists $g'_x \in G$ such that $g'^{-1}_x l''_x \in L_{\theta_x}$ and $h''_x g'_x \in P_{-\theta_x}$. The existence of the limit
\begin{align*}
    g'^{-1}_x (\lim_{t \rightarrow 0}  \lambda_0(t) l''_x h''_x \lambda^{-1}_0(t)) g'_x = \lim_{t \rightarrow 0} (g'^{-1}_x \lambda_0(t) g'_x) (g'^{-1}_x l''_x h''_x g'_x) (g'^{-1}_x \lambda^{-1}_0(t) g'_x)
\end{align*}
shows that $g'^{-1}_x \lambda_0(t) g'_x$ is a cocharacter of some maximal torus in $L_{\theta_x}$. We choose $l_x \in L_{\theta_x}$ and define $g_x: = g'_x l_x$ such that $g_x^{-1} \lambda_0(t) g_x$ and $\lambda_x(t)$ are cocharacters of the same maximal torus in $L_{\theta_x}$. Then we consider the limits
\begin{align*}
    & g^{-1}_x (\lim_{t \rightarrow 0} \lambda_0(t) l''_x \lambda^{-1}_x(t) ) = \lim_{t \rightarrow 0} (g^{-1}_x \lambda_0(t) g_x) (g^{-1}_x l''_x) \lambda^{-1}_x(t) \\
    & (\lim_{t \rightarrow 0} \lambda^{-1}_x(t) h''_x \lambda_0^{-1}(t)) g_x = \lim_{t \rightarrow 0} \lambda^{-1}_x(t) (h''_x g_x) (g^{-1}_x \lambda^{-1}_0(t) g_x).
\end{align*}
Note that
\begin{align*}
    g^{-1}_x l''_x = l^{-1}_x g'^{-1}_x l''_x \in L_{\theta_x}, \ h''_x g_x = h''_x g'_x l_x \in P_{-\theta_x},
\end{align*}
and $g^{-1}_x \lambda_0(t) g_x$ and $\lambda_x(t)$ are cocharacters of the same maximal torus in $L_{\theta_x}$. Therefore, the existence of the above two limits imply $g^{-1}_x \lambda_0(t) g_x = \lambda_x(t)$.
\end{rem}

\begin{prop}\label{prop_stab_equiv}
Given a point $\rho'' \in {\rm Hom}_{\mathbb{S}}(\Pi'',\boldsymbol\theta)$, denote by $\rho'$ the corresponding $\boldsymbol\theta$-filtered Stokes $G$-representation. Then, $\rho'$ is $R$-semistable (resp. $R$-stable) of degree zero if and only if the point $\rho''$ is $\chi_{\boldsymbol\theta}$-semistable (resp. $\chi_{\boldsymbol\theta}$-stable). Moreover, two $\chi_{\boldsymbol\theta}$-semistable points $\rho''_1$ and $\rho''_2$ are GIT equivalent if and only if the corresponding $R$-semistable $\boldsymbol\theta$-filtered Stokes $G$-representations $\rho'_1$ and $\rho'_2$ are $S$-equivalent.
\end{prop}

\begin{proof}
We suppose that $G$ is semisimple first. We regard $\rho''$ as a tuple
\begin{align*}
    ( (a''_i,b''_i)_{1 \leq i \leq g}, (l''_x, h''_x, S''_{x,d} )_{x \in \boldsymbol{D}, d \in \mathbb{A}_{x}} ) \in {\rm Hom}_{\mathbb{S}}(\Pi'',\boldsymbol\theta).
\end{align*}
Suppose that the point $\rho''$ is $\chi_{\boldsymbol\theta}$-semistable. We choose a parabolic subgroup $P$ compatible with $\rho'$ and pick an arbitrary anti-dominant character $\chi$ of $P$, which is trivial on the center of $P$. We have
\begin{equation*}
    \langle \boldsymbol\theta, \chi \rangle = \frac{1}{d} \sum_{x \in \boldsymbol{D}} \langle d \theta_x, \chi \rangle =  \frac{1}{d} \sum_{x \in \boldsymbol{D}} \langle -d \theta_x, -\chi \rangle = \frac{1}{d} \sum_{x \in \boldsymbol{D}} \langle \lambda_{ - \chi} , \chi_{- d \theta_x} \rangle = \frac{1}{d} \langle \lambda_{-\chi}, \chi_{\boldsymbol\theta} \rangle.
\end{equation*}
The cocharacter $\lambda_{-\chi}$ and the element $\rho''$ determine a cocharacter $\lambda: \mathbb{G}_m \rightarrow G_{\boldsymbol\theta}$ such that
\begin{align*}
    \lambda = (\lambda_0, \lambda_x, x \in \boldsymbol{D}), \ \lambda_0:= \lambda_{-\chi}, \ \lambda_x(t) := g^{-1}_x \lambda_0(t) g_x,
\end{align*}
where $g_x$ is given in Remark \ref{rem_cochar}. By Lemma \ref{lem_char_cochar}, we have $P_{\lambda_{-\chi}} \supseteq P$. By the compatibility of $P$ with $\rho'$, the limit $\lim_{t \rightarrow 0} \lambda(t) \cdot \rho''$ exists. Since $\rho''$ is  $\chi_{\boldsymbol\theta}$-semistable by assumption, we have $\langle \lambda_{-\chi}, \chi_{\boldsymbol\theta} \rangle \geq 0$ by Lemma \ref{lem_King_prop2.5}, and thus $\langle \boldsymbol\theta, \chi \rangle \geq 0$. Therefore, $\rho'$ is $R$-semistable.

To prove that $\rho'$ is of degree zero, we suppose that $\rho''$ is $\chi_{\boldsymbol\theta}$-stable and $\rho'$ is $R$-stable for convenience. Given a character $\chi: G \rightarrow \mathbb{G}_m$, the corresponding cocharacter $\lambda_{\chi}$ has the property that its image is in the center of $G$, and thus in $\Delta$. In this case, we always have $\langle \lambda_{\chi}, \chi_{\boldsymbol\theta} \rangle = 0$. Then the formula \eqref{eq_lambda_chi} implies
\begin{align*}
\langle \boldsymbol\theta, \chi \rangle = -d \langle \lambda_{\chi} , \chi_{\boldsymbol\theta}  \rangle = 0.
\end{align*}
Therefore, $\rho'$ is of degree zero.

For the other direction, we suppose that $\rho'$ is $R$-semistable of degree zero. Clearly, $\chi_{\boldsymbol\theta}(\Delta)=\{1\}$ because $\rho'$ is of degree zero. We take a cocharacter $\lambda: \mathbb{G}_m \rightarrow G \times \boldsymbol{L}$ such that the limit $\lim\limits_{t \rightarrow 0} \lambda(t) \cdot \rho''$ exists. Remark \ref{rem_cochar} shows that
\begin{align*}
    a''_i, \, b''_i, \, S''_{x,j}, \, l''_x h''_x \in P_{\lambda_0}.
\end{align*}
Therefore, $P_{\lambda_0}$ is compatible with $\rho'$. Also, $\chi_{\lambda_0}$ is a dominant character of $P_{\lambda_0}$ by Lemma \ref{lem_char_cochar}. Then the formula \eqref{eq_lambda_chi} gives
\begin{align*}
\langle \lambda, \chi_{\boldsymbol\theta} \rangle =  -d \langle \boldsymbol\theta , \chi_{\lambda_0} \rangle \geq 0
\end{align*}
because $\rho'$ is $R$-semistable. This finishes the proof for the semistable case. The argument for the stable case is similar.

When $G$ is reductive, let $R(G)$ be its radical. By Remark \ref{rem_cochar}, a cocharacter $\lambda : \mathbb{G}_m \rightarrow G \times \boldsymbol{L}$, of which the limit $\lim\limits_{t \rightarrow 0} \lambda(t) \cdot \phi$ exists, is uniquely determined by a cocharacter $\lambda_0: \mathbb{G}_m \rightarrow G$. Moreover, $\lambda_0$ is uniquely determined by a cocharacter $\lambda_{ss}$ of the semisimple group $[G,G]$ and a cocharacter $\lambda_{R(G)}$ of $R(G)$. Thus, the formula \eqref{eq_lambda_chi} gives
\begin{align*}
\langle \lambda, \chi_{\boldsymbol\theta} \rangle = -d \langle \boldsymbol\theta , \chi_{\lambda_0} \rangle = -d (\langle \boldsymbol\theta , \chi_{\lambda_{ss}} \rangle + \langle \boldsymbol\theta , \chi_{\lambda_{R(G)}} \rangle ).
\end{align*}
Note that there is a natural injection of characters
\begin{align*}
    {\rm Hom}(G, \mathbb{G}_m) \rightarrow {\rm Hom}(R(G),\mathbb{G}_m),
\end{align*}
whose image is of finite index. Therefore, ${\rm Hom}(G, \mathbb{G}_m) \otimes_{\mathbb{Z}} \mathbb{Q} \cong {\rm Hom}(R(G), \mathbb{G}_m) \otimes_{\mathbb{Z}} \mathbb{Q}$. By definition of degree zero, we know that for any character $\chi$ of $G$, we have $\langle \boldsymbol\theta, \chi \rangle =0$. Then, $\langle \boldsymbol\theta , \chi_{\lambda_{R(G)}} \rangle = 0$. Therefore,
\begin{align*}
    \langle \lambda, \chi_{\boldsymbol\theta} \rangle = -d \langle \boldsymbol\theta , \chi_{\lambda_{ss}} \rangle = \langle \lambda_{ss}, \chi_{\boldsymbol\theta} \rangle,
\end{align*}
and we reduces it to the semisimple case.

\hspace{\fill}

For the second statement about $S$-equivalence, we introduce some notations first. Given $\rho'' \in {\rm Hom}_{\mathbb{S}}(\Pi'',\boldsymbol\theta)$, let $\lambda: \mathbb{G}_m \rightarrow G \times \boldsymbol{L}$ be a cocharacter given by a tuple $(\lambda_0, \lambda_x, x \in \boldsymbol{D})$, where $\lambda_0: \mathbb{G}_m \rightarrow G$ and $\lambda_x : \mathbb{G}_m \rightarrow L_{\theta_x}$ are cocharacters, such that the limit $\lim_{t \rightarrow 0} \lambda(t) \cdot \rho''$ exists. Denote by $P$ the parabolic subgroup given by $\lambda_0$ with Levi subgroup $L$. We introduce the following notations:
\begin{align*}
    & (a''_i)_{L} := \lim_{t \rightarrow 0} \lambda_0(t) a''_i \lambda^{-1}_0(t), \
    (b''_i)_{L} := \lim_{t \rightarrow 0} \lambda_0(t) b''_i \lambda^{-1}_0(t), \
    (S''_{x,d})_{L} := \lim_{t \rightarrow 0} \lambda_0(t) S''_{x,d} \lambda^{-1}_0(t), \\
    & (l''_x)_{L} := \lim_{t \rightarrow 0} \lambda_0(t) l''_x \lambda^{-1}_x(t), \
    (h''_x)_{L} := \lim_{t \rightarrow 0} \lambda_x(t) h''_x \lambda^{-1}_0(t), \
    (\rho'')_{L} := \lim_{t \rightarrow 0} \lambda(t) \cdot \rho'',
\end{align*}
and clearly,
\begin{align*}
    (a''_i)_L, \ (b''_i)_L, \ (S''_{x,d})_L, \ (l''_x)_L (h''_x)_L \in L.
\end{align*}
Denote by $\rho': \Pi' \rightarrow G$ the corresponding representation of $\rho''$ under the morphism $\Pi' \rightarrow \Pi''$, and then $(\rho')_L$ is the corresponding representation of $(\rho'')_L$. In the following, we suppose that $G$ is semisimple, and the reductive case can be reduced to the semisimple case as we discussed above.

Assume that $\rho''_1$ and $\rho''_2$ are GIT equivalent. By Lemma \ref{lem_King_prop2.5}, there exist cocharacters $\lambda_1$ and $\lambda_2$ of $G \times \boldsymbol{L}$ such that
\begin{itemize}
    \item $\langle \lambda_1 , \chi_{\boldsymbol\theta} \rangle = \langle \lambda_2 , \chi_{\boldsymbol\theta} \rangle = 0$,
    \item the limits $(\rho''_1)_{L_1} = \lim\limits_{t \rightarrow 0} \lambda_1(t) \cdot \rho''_1$ and $(\rho''_2)_{L_2} = \lim\limits_{t \rightarrow 0} \lambda_2(t) \cdot \rho''_2$ exist,
    \item $(\rho''_1)_{L_1}$ and $(\rho''_2)_{L_2}$ are in the same $(G \times \boldsymbol{L})$-orbit.
\end{itemize}
There exists $(g, (l_x)_{x \in \boldsymbol{D}}) \in G \times \boldsymbol{L}$ such that
\begin{align*}
    (g, (l_x)_{x \in \boldsymbol{D}}) \cdot (\lim\limits_{t \rightarrow 0} \lambda_1(t) \cdot \rho''_1 ) = \lim\limits_{t \rightarrow 0} \lambda_2(t) \cdot \rho''_2.
\end{align*}
By Lemma \ref{lem_corr_pi'_pi''}, we have
\begin{align*}
    g \cdot (\lim\limits_{t \rightarrow 0} \lambda_{10}(t) \cdot \rho'_1) = \lim\limits_{t \rightarrow 0} \lambda_{20}(t) \cdot \rho'_2.
\end{align*}
The cocharacters $\lambda_{10}$ and $\lambda_{20}$ determine parabolic subgroups $P_1$ and $P_2$ and Levi subgroups $L_1$ and $L_2$ respectively. Clearly,
\begin{align*}
    \lim\limits_{t \rightarrow 0} \lambda_{i0}(t) \cdot \rho'_i = (\rho'_i)_{L_i}, \ i=1,2 \ ,
\end{align*}
and thus
\begin{align*}
    g \cdot (\rho'_1)_{L_1} = (\rho'_2)_{L_2}.
\end{align*}
Now we have to prove that $\rho'_i$ is admissible with $P_i$ based on the condition $\langle \lambda_i, \chi_{\boldsymbol\theta} \rangle = -d \langle \boldsymbol\theta , \chi_{\lambda_{i0}} \rangle = 0$. It is equivalent to show that for any character $\chi_i: P_i \rightarrow \mathbb{G}_m$, we have $\langle
\boldsymbol\theta, \chi_i \rangle = 0$. With the same approach as in \cite{Ram96a} (for instance the proof of \cite[Lemma 3.5.8]{Ram96a}), it is equivalent to choose a faithful embedding $G \hookrightarrow {\rm GL}(V)$ and prove this property for general linear groups. Therefore, the argument can be proved in the same way as \cite[Lemma 3.22]{HS23}. In conclusion, $\rho'_1$ and $\rho'_2$ are $S$-equivalent.

For the other direction, suppose that $\rho'_1$ and $\rho'_2$ are $S$-equivalent, and then there exist parabolic subgroups $P_1$ and $P_2$ (with Levi subgroups $L_1$ and $L_2$) admissible with $\rho'_1$ and $\rho'_2$ respectively such that $g \cdot (\rho'_1)_{L_1} = (\rho'_2)_{L_2}$ for some $g \in G$. Clearly, $gP_1 g^{-1} = P_2$. We choose cocharacters $\lambda_{i0}: \mathbb{G}_m \rightarrow G$ such that $P_{\lambda_{i0}} = P_i$ for $i=1,2$. We define cocharacters $\lambda_i: \mathbb{G}_m \rightarrow G \times \boldsymbol{L}$ as
\begin{align*}
    \lambda_i = (\lambda_{i0}, \lambda_{ix}, x \in \boldsymbol{D}),
\end{align*}
where $\lambda_{ix}: = g^{-1}_{ix} \lambda_{i0} g_{ix}: \mathbb{G}_m \rightarrow L_{\theta_x}$ and $g_{ix} \in G$ is given in Remark \ref{rem_cochar}. Since $P_i$ is compatible with $\rho'_i$, we have
\begin{align*}
    \langle \lambda_i, \chi_{\boldsymbol\theta} \rangle = -d \langle \boldsymbol\theta , \chi_{\lambda_{i0}} \rangle = 0
\end{align*}
and the limit
\begin{align*}
    \lim_{t \rightarrow 0} \lambda_i(t) \cdot \rho''_i
\end{align*}
exist for $i =1,2$. By Lemma \ref{lem_King_prop2.5}, we only have to prove that $(\rho''_1)_{L_1}$ and $(\rho''_2)_{L_2}$ are in the same $G_{\boldsymbol\theta}$-orbit. Since $g \cdot (\rho'_1)_{L_1} = (\rho'_2)_{L_2}$, the key point is to find $l_x \in L_{\theta_x}$ for each $x \in \boldsymbol{D}$ such that
\begin{align*}
    (g , (l_x)_{x \in \boldsymbol{D}}) \cdot (\rho''_1)_{L_1} = (\rho''_2)_{L_2}.
\end{align*}
Consider the element
\begin{align*}
    g^{-1}_{ix} (l''_{ix})_{L_i} = \lim_{t \rightarrow 0} (g^{-1}_{ix} \lambda_{i0}(t) g_{ix}) (g^{-1}_{ix} l''_{ix}) \lambda^{-1}_{ix}(t).
\end{align*}
Since $g^{-1}_{ix} \lambda_{i0}(t) g_{ix} = \lambda_{ix}(t)$ and $g^{-1}_{ix} l''_{ix} \in L_{\theta_x}$, we have $g^{-1}_{ix} (l''_{ix})_{L_i} \in L_{\theta_x}$. With a similar argument, we have $(h''_{ix})_{L_i} g_{ix} \in L_{\theta_x}$. Note that
\begin{align*}
    g (l''_{1x})_{L_1}(h''_{1x})_{L_1} g^{-1} = (l''_{2x})_{L_2}(h''_{2x})_{L_2}.
\end{align*}
Reformulating the equation, we have
\begin{align*}
    g^{-1}_{2x} g g_{1x} ( g^{-1}_{1x} (l''_{1x})_{L_1} ) (( h''_{1x})_{L_1} g_{1x} ) g^{-1}_{1x} g^{-1} g_{2x}= ( g^{-1}_{2x} (l''_{2x})_{L_2} ) ((h''_{2x})_{L_2} g_{2x}).
\end{align*}
Therefore, $g^{-1}_{2x} g g_{1x} \in L_{\theta_x}$ because the normalizer of $L_{\theta_x}$ is itself. We define
\begin{align*}
    l_x := ((l''_{2x})_{L_2})^{-1} g (l''_{1x})_{L_1}
\end{align*}
Clearly,
\begin{align*}
    l_x = (g^{-1}_{2x}(l''_{2x})_{L_2})^{-1} (g^{-1}_{2x} g g_{1x}) (g^{-1}_{1x} (l''_{1x})_{L_1}) \in L_{\theta_x}.
\end{align*}
Since
\begin{align*}
    l_x =  ((l''_{2x})_{L_2})^{-1} g (l''_{1x})_{L_1} = (h''_{2x})_{L_2} g ((h''_{1x})_{L_1})^{-1},
\end{align*}
it is easy to check
\begin{align*}
    (g , (l_x)_{x \in \boldsymbol{D}}) \cdot (\rho''_1)_{L_1} = (\rho''_2)_{L_2}.
\end{align*}
Therefore, $(\rho''_1)_{L_1}$ and $(\rho''_2)_{L_2}$ are in the same $G_{\boldsymbol\theta}$-orbit.
\end{proof}

Under the equivalence of Stokes $G$-representations and Stokes $G$-local systems (Corollary \ref{cor_Stokes_rep_and_loc}), we obtain the moduli space of filtered Stokes $G$-local systems.

\begin{thm}\label{thm_Sto_moduli}
The quasi-projective variety
\begin{align*}
\mathcal{M}_{\rm B}(X_{\boldsymbol{D}},G,\boldsymbol{Q},\boldsymbol\theta) := {\rm Hom}_{\mathbb{S}}(\Pi'',\boldsymbol\theta) /\!\!/ (G \times \boldsymbol{L}, \chi_{\boldsymbol\theta})
\end{align*}
is the moduli space of degree zero $R$-semistable $\boldsymbol\theta$-filtered Stokes $G$-local systems with irregular type $\boldsymbol{Q}$ on $X_{\boldsymbol{D}}$, of which points are in one-to-one correspondence with $S$-equivalence classes of degree zero $R$-semistable $\boldsymbol\theta$-filtered Stokes $G$-local systems with irregular type $\boldsymbol{Q}$. There exists an open subset $\mathcal{M}^s_{\rm B}(X_{\boldsymbol{D}},G,\boldsymbol{Q},\boldsymbol\theta)$, of which points correspond to isomorphism classes of degree zero $R$-stable $\boldsymbol\theta$-filtered Stokes $G$-local systems with irregular type $\boldsymbol{Q}$.
\end{thm}

\begin{proof}
By Proposition \ref{prop_stab_equiv}, the theorem follows directly from King's result \cite[\S 2]{King94}.
\end{proof}

\section{Examples}\label{sect_eg}

We study some examples of the moduli space of filtered Stokes $G$-local systems in this section. In \S\ref{subsect_trivial_weight}, we consider the case of trivial weights. We conclude that the $R$-stability of filtered Stokes $G$-representations (with trivial weights) is equivalent to the irreducibility of the corresponding representations (Lemma \ref{lem_trivial_weight}) and the moduli space is an affine variety, which is known as the \emph{wild character variety} \cite{Boa14,DDP18}. In \S\ref{subsect_trivial_irrg}, we consider the case of trivial irregular types. If all irregular types are trivial, then the discussion completely reduces to the tame case, of which the moduli space has been constructed in \cite{HS23}. In \S\ref{subsect_EH_space}, we consider the Eguchi--Hanson space, of which the irregular type is unramified. In this case, we find a particular $\theta$-filtered Stokes $G$-local system, which is $R$-stable but not semisimple as a representation. This example shows that wild character varieties may not be the Betti moduli space in the nonabelian Hodge correspondence. In \S\ref{subsect_Airy} and \S\ref{subsect_sl2_ramified}, we start from the Airy equation and study Stokes ${\rm SL}_2(\mathbb{C})$-local systems with ramified irregular type on $X_{\boldsymbol{D}}$, where $(X,\boldsymbol{D}) = (\mathbb{P}^1, 0)$. In this case, we find that Stokes ${\rm SL}_2(\mathbb{C})$-representations are always irreducible. Therefore, the corresponding moduli space is exactly the wild character variety by Corollary \ref{cor_wild_char_var}.

\subsection{Trivial weights}\label{subsect_trivial_weight}
Fixing a collection of irregular types $\boldsymbol{Q}$, let $\Pi'$ be the group defined in \S\ref{subsect_Sto_G_rep}. Suppose that all weights are trivial, i.e. $\theta_x=0$ for $x \in \boldsymbol{D}$. We use the notation $\boldsymbol{0}$ for the collection of trivial weights. In this case, $L_{\theta_x} = G$. Then,
\begin{align*}
    G \times \boldsymbol{L} = G \times \prod_{x \in \boldsymbol{D}} G.
\end{align*}
Moreover, the character $\chi_{-d \theta_x}$ is also trivial, which implies that the character $\chi_{\boldsymbol{0}}: G \times \boldsymbol{L} \rightarrow \mathbb{G}_m$ is trivial.

\begin{lem}\label{lem_trivial_weight}
A degree zero $\boldsymbol{0}$-filtered Stokes $G$-representation $\rho': \Pi' \rightarrow G$ is $R$-stable (resp. $R$-semistable) if and only if it is an irreducible (resp. semisimple) representation.
\end{lem}

\begin{proof}
Given an irreducible representation $\rho': \Pi' \rightarrow G$, it cannot be restricted to any proper nontrivial parabolic subgroup $P$. Then it is $R$-stable automatically. On the other hand, given a degree zero $R$-stable $\boldsymbol{0}$-filtered Stokes $G$-representation $\rho'$, suppose that a nontrivial proper parabolic subgroup $P$ is compatible with $\rho'$. We choose an arbitrary nontrivial anti-dominant character $\chi: P \rightarrow \mathbb{G}_m$, which is trivial on the center of $P$. We have
\begin{align*}
    \deg^{\rm loc} \rho'(P,\chi) = \langle \boldsymbol{0}, \chi \rangle = 0
\end{align*}
because $\theta_x$ is trivial for every $x \in \boldsymbol{D}$. This contradicts the assumption that $\rho'$ is $R$-stable. Therefore, any nontrivial proper parabolic subgroup is not compatible with $\rho'$, which means that $\rho'$ is irreducible. The proof for semistable case is similar.
\end{proof}

\begin{cor}\label{cor_wild_char_var}
The moduli space
\begin{align*}
    \mathcal{M}_{\rm B}(X_{\boldsymbol{D}},G,\boldsymbol{Q},\boldsymbol{0}) = {\rm Hom}_{\mathbb{S}}(\Pi'',\boldsymbol{0}) /\!\!/ (G \times \boldsymbol{L},\chi_{\boldsymbol{0}})
\end{align*}
is an affine variety, of which points correspond to isomorphism classes of semisimple Stokes $G$-representations. There exists an open subset $\mathcal{M}^s_{\rm B}(X_{\boldsymbol{D}},G,\boldsymbol{Q},\boldsymbol{0})$, of which points correspond to isomorphism classes of irreducible Stokes $G$-representations.
\end{cor}

\begin{proof}
Since all weights $\theta_x$ are trivial, ${\rm Hom}_{\mathbb{S}}(\Pi'',\boldsymbol{0}) = {\rm Hom}_{\mathbb{S}}(\Pi'',G)$ is an affine variety. Since the character $\chi_{\boldsymbol{0}}$ is trivial, the GIT quotient ${\rm Hom}_{\mathbb{S}}(\Pi'',\boldsymbol{0}) /\!\!/ (G \times \boldsymbol{L},\chi_{\boldsymbol{0}})$ is an affine variety. By Theorem \ref{thm_Sto_moduli}, points in $\mathcal{M}_{\rm B}(X,G,\boldsymbol{Q},\boldsymbol{0})$ are in one-to-one correspondence with $S$-equivalence classes of degree zero $R$-semistable $\boldsymbol{0}$-filtered Stokes $G$-representations, which are exactly isomorphism classes of semisimple representations by Lemma \ref{lem_trivial_weight}. This finishes the proof of this corollary.
\end{proof}

The moduli space $\mathcal{M}_{\rm B}(X,G,\boldsymbol{Q},\boldsymbol{0})$ given above is exactly the wild character variety considered by Boalch \cite[\S8 and \S9]{Boa14}. The difference is that Boalch constructed the moduli space from the fundamental groupoid $\Pi$, while we construct the moduli space from $\Pi'$ and $\Pi''$. Moreover, we also refer the reader to \cite[Theorem 9.3]{Boa14} for another proof that $\rho \in {\rm Hom}_{\mathbb{S}}(\Pi,G)$ is stable (in the sense of GIT) if and only if it is irreducible.

\subsection{Trivial Irregular Types}\label{subsect_trivial_irrg}
Suppose that all irregular types are trivial, i.e. $Q_x=0$ for any $x\in\boldsymbol{D}$, and we use the notation $\boldsymbol{0}$ for the collection of trivial irregular types. In this case, the Betti moduli space in the wild case completely reduce to the tame case considered in \cite{HS23}.

Recall that the generators of the fundamental groupoid $\Pi$ of $X_{\boldsymbol{Q}}$ introduced in \S\ref{subsect_glo_corr} are
given by
\begin{enumerate}
    \item $\alpha_i,\beta_i$, $1 \leq i \leq g$;
    \item $\gamma_x$ for $x \in \boldsymbol{D}$;
    \item $\gamma_{x,d}$ for $x \in \boldsymbol{D}$ and $d \in \mathbb{A}_x$;
    \item $\gamma_{0x}$ for $x \in \boldsymbol{D}$.
\end{enumerate}
Since all irregular types are trivial, there is no anti-Stokes directions and then the set $\mathbb{A}_x$ is empty for every $x \in \boldsymbol{D}$. Therefore, when irregular types are trivial, the group $\Pi$ is generated by $\alpha_i, \beta_i, \gamma_x, \gamma_{0x}$ with the relation
\begin{align*}
    \left(  \prod_{i=1}^g [\alpha_i, \beta_i] \right) \cdot \left( \prod_{x \in \boldsymbol{D}} \mu_x \right) = {\rm id},
\end{align*}
where $\mu_x = \gamma^{-1}_{0x} \cdot \gamma_x \cdot \gamma_{0x}$. Following the same discussion, the group $\Pi'$ introduced in \S\ref{subsect_Sto_G_rep} is generated by
\begin{enumerate}
    \item $\alpha'_i,\beta'_i$ for $1 \leq i \leq g$;
    \item $\gamma'_x$ for $x \in \boldsymbol{D}$,
\end{enumerate}
with the relation
\begin{align*}
    \left(  \prod_{i=1}^g [\alpha'_i, \beta'_i] \right) \cdot \left( \prod_{x \in \boldsymbol{D}} \mu'_x \right) = {\rm id},
\end{align*}
where $\mu'_x = \gamma'_x$. Clearly, $\Pi'$ can be regarded as the fundamental group of $X_{\boldsymbol{D}}$ and the formal monodromy reduces to the topological monodromy. By Proposition \ref{prop_reps}, isomorphism classes of Stokes $G$-local systems with irregular type $\boldsymbol{0}$ are in one-to-one correspondence with $G$-orbits in ${\rm Hom}(\Pi',G)$, which implies that Stokes $G$-local systems with irregular type $\boldsymbol{0}$ on $X_{\boldsymbol{D}}$ are exactly $G$-local systems on $X_{\boldsymbol{D}}$. For the construction of the moduli space, we introduce the third group $\Pi''$ in \S\ref{subsect_cons_pi''}, which is generated by
\begin{enumerate}
    \item $\alpha''_i, \beta''_i$ for $1 \leq i \leq g$;
    \item $\iota''_x,\gamma''_x$ for $x \in \boldsymbol{D}$
\end{enumerate}
with the relation
\begin{align*}
    \left(  \prod_{i=1}^g [\alpha''_i, \beta''_i] \right) \cdot \left( \prod_{x \in \boldsymbol{D}} \mu''_x \right) = {\rm id},
\end{align*}
where $\mu''_x = \iota''_x \gamma''_x$. Fixing a collection of weights $\boldsymbol\theta = \{\theta_x, x \in \boldsymbol{D}\}$, the variety ${\rm Hom}_{\mathbb{S}}(\Pi'',\boldsymbol\theta)$ now parametrizes points
\begin{align*}
    ( (a''_i,b''_i)_{1 \leq i \leq g} , (l''_x,h''_x)_{x \in \boldsymbol{D}} )
\end{align*}
such that for each $x \in \boldsymbol{D}$, there exists $g_x \in G$ such that $g^{-1}_x l''_x \in L_{\theta_x}$ and $h''_x g_x \in P_{\theta_x}$. Moreover, there is a natural $(G \times \boldsymbol{L})$-action on ${\rm Hom}_{\mathbb{S}}(\Pi'',\boldsymbol\theta)$. This construction is exactly the same as \cite[Construction 3.10]{HS23}, and it is easy to check that the stability condition in this special case is equivalent to that in \cite{HS23}. Therefore, the moduli space $\mathcal{M}_{\rm B}(X,G,\boldsymbol{0},\boldsymbol\theta)$ of degree zero $R$-semistable $\boldsymbol\theta$-filtered Stokes $G$-local systems with irregular type $\boldsymbol{0}$ on $X_{\boldsymbol{D}}$ is exactly the moduli space $\mathcal{M}_{\rm B}(X_{\boldsymbol{D}}, G ,\boldsymbol\theta)$ of degree zero $R$-semistable $\boldsymbol\theta$-filtered $G$-local systems on $X_{\boldsymbol{D}}$ \cite[Theorem 1.2]{HS23}.

\subsection{``Weighted'' Eguchi--Hanson Space}\label{subsect_EH_space}
We consider an explicit example such that $G=\mathrm{SL}_2(\mathbb{C})$ and $(X,\boldsymbol{D})=(\mathbb{P}^1,0)$. Denote by $\alpha$ and $-\alpha$ the roots of ${\rm SL}_2(\mathbb{C})$, and let
\begin{align*}
U_+ := U_\alpha =
\begin{pmatrix}
1 & \ast \\
0 & 1
\end{pmatrix}, \
U_- := U_{-\alpha} = \begin{pmatrix}
    1 & 0\\
    \ast & 1
\end{pmatrix}.
\end{align*}
Given an irregular type with a pole of order $3$ at $z=0$
\begin{align*}
Q_{-3}=\frac{A_3}{z^3}+\frac{A_2}{z^2}+\frac{A_1}{z},
\end{align*}
where the subscript of the irregular type is for its degree, the leading coefficient $A_3$ is nontrivial. Since $G={\rm SL}_2(\mathbb{C})$, the leading coefficient $A_3$ is automatically regular and semisimple. In this case, the irregular type $Q_{-3}$ has $6$ anti-Stokes directions and \begin{align*}
    \mathbb{S}{\rm to}(Q_{-3}) = (U_+ \times U_-)^3.
\end{align*}
Since this irregular type $Q_{-3}$ is in the unramified case, the centralizer of $Q_{-3}$ coincides with the set of formal monodromies given by $Q_{-3}$, i.e.
\begin{align*}
H= H(\partial)=
\left\{\begin{pmatrix}
a&\\
&a^{-1}
\end{pmatrix}\ \Big| \ a\in\mathbb{C}^*\right\}.
\end{align*}
Therefore, the space ${\rm Hom}_{\mathbb{S}}(\Pi_1(X_{Q_{-3}}), G) \subseteq G \times H \times (U_+ \times U_-)^3$ is a closed subvariety including
\begin{align*}
    (c,h,(u_{+,i},u_{-,i})_{1 \leq i \leq 3}) \in G \times H \times (U_+ \times U_-)^3
\end{align*}
such that $c^{-1} (h \prod_{i=1}^3 (u_{+,i} u_{-,i})) c = {\rm id}$.

Given the weight
\begin{align*}
\theta=
\begin{pmatrix}
\frac{1}{2}&\\
& -\frac{1}{2}
\end{pmatrix}\in\mathfrak{t},
\end{align*}
it determines a parabolic subgroup of $\mathrm{SL}_2(\mathbb{C})$ as
\begin{align*}
P_\theta=
\left\{\begin{pmatrix}
a&b\\
0&a^{-1}
\end{pmatrix}\ | \ a\in\mathbb{C}^*, b\in\mathbb{C}\right\}.
\end{align*}
We consider a special $\theta$-filtered Stokes $G$-representation
\begin{align*}
    \rho = (c,h,(u_{+,i},u_{-,i})_{1 \leq i \leq 3}) \in {\rm Hom}_{\mathbb{S}}(\Pi_1(X_{Q_{-3}}), G)
\end{align*}
such that $c,h,u_{+,i},u_{-,i} \in P_\theta$ and at least one of them is not included in $P_{-\theta}$. For example,
\begin{align*}
u_{+,1} = \begin{pmatrix}
1 & 1\\
0 & 1
\end{pmatrix}, \
u_{+,2} = \begin{pmatrix}
1 & -1\\
0 & 1
\end{pmatrix}
\end{align*}
and the other elements are identity matrix. Clearly, if we take $P=P_\theta$, the parabolic subgroup $P$ is compatible with $\rho$. Furthermore, taking any nontrivial anti-dominant character $\chi$ of $P$, it is easy to check
\begin{align*}
    \deg^{\rm loc} \rho(P,\chi) = \langle \theta, \chi \rangle > 0,
\end{align*}
and the $\theta$-filtered Stokes $G$-representation $\rho$ is $R$-stable. By Theorem \ref{thm_Sto_moduli}, it corresponds to a point in $\mathcal{M}^s_{\rm B}(X_{\boldsymbol{D}},G,Q_{-3},\theta)$.

On the other hand, this representation $\rho$ is indecomposable but not semisimple. As we discussed in \S\ref{subsect_trivial_weight}, the classical wild character variety $\mathcal{M}_{\rm B}(X_{\boldsymbol{D}},G,Q,0)$ only parametrizes semisimple representations. Therefore, the $\boldsymbol\theta$-filtered Stokes $G$-representation $\rho$ does not correspond to a point in $\mathcal{M}^s_{\rm B}(X_{\boldsymbol{D}},G,Q_{-3},0)$. This example illustrates the fact that
\begin{align*}
    \mathcal{M}_{\rm B}(X_{\boldsymbol{D}},{\rm SL}_2(\mathbb{C}),Q_{-3},0) \ncong \mathcal{M}_{\rm B}(X_{\boldsymbol{D}},{\rm SL}_2(\mathbb{C}),Q_{-3},\theta)
\end{align*}
for general weight $\theta$. Moreover, the moduli space $\mathcal{M}_{\rm B}(X_{\boldsymbol{D}},G,Q_{-3},0)$ is exactly the Eguchi--Hanson space considered in \cite[\S2]{Boa18}, where he concluded that this is a smooth affine variety of dimension $2$.

\subsection{Airy Equation}\label{subsect_Airy}

We start from the Airy equation
\begin{align*}
    y''(t)=ty(t).
\end{align*}
This equation corresponds to the following connection
\begin{align*}
\nabla=d+
\begin{pmatrix}
0&-t\\
-1&0
\end{pmatrix}dt.
\end{align*}
We change the coordinate $z=\frac{1}{t}$
\begin{align*}
\nabla = d+
\begin{pmatrix}
0&z^{-3}\\
z^{-2}&0
\end{pmatrix}dz.
\end{align*}
Clearly, $\nabla$ has an irregular singularity at $z=0$ (or $t=\infty$). Thus, the Airy equation corresponds to a connection on $\mathbb{P}^1 \backslash \{ 0 \}$ with an irregular singularity at $z=0$, at which the topological monodromy is trivial \cite[Example 8.15]{vdPS03}. We continue working locally on the local coordinate $z$. Let
\begin{align*}
A(z) =
\begin{pmatrix}
0&z^{-3}\\
z^{-2}&0
\end{pmatrix} =
\begin{pmatrix}
0 & 1\\
0 & 0
\end{pmatrix} z^{-3} +
\begin{pmatrix}
0 & 0\\
1 & 0
\end{pmatrix} z^{-2}
\end{align*}
be the connection form of $\nabla$. We take (on a ramified cover)
\begin{align*}
    g_1=\begin{pmatrix}
        z^{1/4}&\\
        &z^{-1/4}
    \end{pmatrix}
\end{align*}
and get
\begin{align*}
\begin{aligned}
    g_1\circ \nabla &=d-dg_1\cdot g_1^{-1} + (g_1 A(z) g_1^{-1}) dz\\
    &=d+\begin{pmatrix}
        0 & z^{-5/2}\\
        z^{-5/2} & 0
    \end{pmatrix}dz-\frac{1}{4}\begin{pmatrix}
        1&0\\
        0&-1
    \end{pmatrix}\frac{dz}{z}
\end{aligned}
\end{align*}
Then we take
\begin{align*}
    g_2=\begin{pmatrix}
        \frac{1}{\sqrt{2}}&-\frac{1}{\sqrt{2}}\\
        \\
        \frac{1}{\sqrt{2}}&\frac{1}{\sqrt{2}}
    \end{pmatrix},
\end{align*}
and obtain
\begin{align*}
\begin{aligned}
    g_2\circ (g_1\circ \nabla)
    &=d+\begin{pmatrix}
        -z^{-5/2} & 0\\
        0 & z^{-5/2}
    \end{pmatrix}dz-\frac{1}{4}\begin{pmatrix}
        0&1\\
        1&0
    \end{pmatrix}\frac{dz}{z}\\
    &=d+d\Big(\begin{pmatrix}
        \frac{2}{3}z^{-3/2}&0\\
        0&-\frac{2}{3}z^{-3/2}
    \end{pmatrix}\Big)-\frac{1}{4}\begin{pmatrix}
        0&1\\
        1&0
    \end{pmatrix}\frac{dz}{z}.
\end{aligned}
\end{align*}
Therefore, the corresponding connection of the Airy equation is of irregular type
\begin{align*}
    Q_{-\frac{3}{2}}= \begin{pmatrix}
        \frac{2}{3}&0\\
        0&-\frac{2}{3}
    \end{pmatrix}z^{-\frac{3}{2}},
\end{align*}
which is in the ramified case. In conclusion, the corresponding connection of the Airy equation is a ${\rm SL}_2(\mathbb{C})$-connection with irregular type $Q_{-\frac{3}{2}}$ on $X_{\boldsymbol{D}}$, where $(X,\boldsymbol{D}) = (\mathbb{P}^1,0)$, and thus corresponds to a Stokes ${\rm SL}_2(\mathbb{C})$-representation (local system) with irregular type $Q_{-\frac{3}{2}}$ on $X_{\boldsymbol{D}}$.

\subsection{Stokes ${\rm SL}_2(\mathbb{C})$-local Systems with Ramified Irregular Types}\label{subsect_sl2_ramified}

We start from the setup
\begin{align*}
    G={\rm SL}_2(\mathbb{C}), \ (X,\boldsymbol{D}) = (\mathbb{P}^1,0), \ Q_{-\frac{3}{2}}=\begin{pmatrix}
        \frac{2}{3}&0\\
        0&-\frac{2}{3}
    \end{pmatrix}z^{-\frac{3}{2}},
\end{align*}
which is the same as \S\ref{subsect_Airy}, and consider Stokes ${\rm SL}_2(\mathbb{C})$-representations with irregular type $Q_{-\frac{3}{2}}$ on $X_{\boldsymbol{D}}$. Following the same notation as in \S\ref{subsect_EH_space}, we have
\begin{align*}
    \mathbb{S}{\rm to}(Q_{-\frac{3}{2}}) = U_- \times U_+ \times U_-
\end{align*}
and
\begin{align*}
    H=
\left\{\begin{pmatrix}
a&\\
&a^{-1}
\end{pmatrix}\ \Big| \ a\in\mathbb{C}^*\right\}, \ H(\partial)=
\left\{\begin{pmatrix}
& a \\
-a^{-1} &
\end{pmatrix}\ \Big| \ a\in\mathbb{C}^*\right\}.
\end{align*}
A point $(c,h, u_{-,1} , u_{+}, u_{-,2}) \in G \times H(\partial) \times U_- \times U_+ \times U_-$ corresponds to a Stokes $G$-representation in ${\rm Hom}_{\mathbb{S}}(\Pi_1(X_{Q_{-\frac{3}{2}}}),G)$ if it satisfies the condition that $c^{-1} h u_{-,1} u_+ u_{-,2} c = {\rm id}$. Ignoring $c$, we consider the equation
\begin{align*}
    h u_{-,1} u_+ u_{-,2} = {\rm id}.
\end{align*}
Let
\begin{align*}
    h = \begin{pmatrix}
        0 & a \\
        -a^{-1} & 0
    \end{pmatrix}, \
    u_{+,1} = \begin{pmatrix}
        1 & 0 \\
        b & 1
    \end{pmatrix}, \
    u_- = \begin{pmatrix}
        1 & c \\
        0 & 1
    \end{pmatrix}, \
    u_{+,2} = \begin{pmatrix}
        1 & 0 \\
        d & 1
    \end{pmatrix}.
\end{align*}
We have
\begin{align*}
    \begin{pmatrix}
        0 & -a \\
        a^{-1} & 0
    \end{pmatrix} =
    \begin{pmatrix}
        1 & 0 \\
        b & 1
    \end{pmatrix}
    \begin{pmatrix}
        1 & c \\
        0 & 1
    \end{pmatrix}
    \begin{pmatrix}
        1 & 0 \\
        d & 1
    \end{pmatrix}.
\end{align*}
Therefore,
\begin{align*}
    b=a^{-1} , \ c=-a, \ d=a^{-1}.
\end{align*}
This means that the formal monodromy uniquely determines Stokes data in this case. Therefore, the moduli space
\begin{align*}
    \mathcal{M}_{\rm B}(X_{\boldsymbol{D}},{\rm SL}_2(\mathbb{C}) ,Q_{-\frac{3}{2}}, 0) \cong {\rm pt}
\end{align*}
is a point. Note that whatever the formal monodromy is, the corresponding representation is irreducible. By the discussion in \S\ref{subsect_trivial_weight}, we have
\begin{align*}
    \mathcal{M}_{\rm B}(X_{\boldsymbol{D}},{\rm SL}_2(\mathbb{C}) ,Q_{-\frac{3}{2}}, \theta) \cong \mathcal{M}_{\rm B}(X_{\boldsymbol{D}},{\rm SL}_2(\mathbb{C}) ,Q_{-\frac{3}{2}}, 0) \cong {\rm pt}
\end{align*}
for any weight $\theta$. Moreover, the Stokes ${\rm SL}_2(\mathbb{C})$-representation corresponding to the Airy equation is the unique point in the moduli space $\mathcal{M}_{\rm B}(X_{\boldsymbol{D}},{\rm SL}_2(\mathbb{C}) ,Q_{-\frac{3}{2}}, \theta)$ as we discussed in \S\ref{subsect_Airy}. As a result, the connection corresponds to the Airy equation is both rigid and physically rigid.

\begin{rem}
Recently, we noticed that Hohl and Jakob studied physical rigidity of Kloosterman connections and applied the result to a $G$-version Airy equation and get the same result for rigidity \cite[Theorem 1.2.1]{HJ24}.
\end{rem}

Now we suppose that $Q$ is a ramified irregular type. Since we work on ${\rm SL}_2(\mathbb{C})$, the leading coefficient of $Q$ is regular and semisimple. Denote by $\mathbb{A}$ the set of anti-Stokes directions of $Q$, and for each $d \in \mathbb{A}$, $\mathbb{S}{\rm to}_d$ is a product of $U_+$ and $U_-$. Let $\mathbb{S}{\rm to}(Q):= \prod_{d \in \mathbb{A}} \mathbb{S}{\rm to}_d$. At the same time, we have
\begin{align*}
    H=
\left\{\begin{pmatrix}
a&\\
&a^{-1}
\end{pmatrix}\ \Big| \ a\in\mathbb{C}^*\right\}, \ H(\partial)=
\left\{\begin{pmatrix}
& a \\
-a^{-1} &
\end{pmatrix}\ \Big| \ a\in\mathbb{C}^*\right\}.
\end{align*}
A Stokes ${\rm SL}_2(\mathbb{C})$-representation with irregular type $Q$ can be regarded as a tuple $(c,h,(S_{d})_{d \in \mathbb{A}})$ such that
\begin{equation}\label{eq_rela_exmp}\tag{$\star$}
   c^{-1} h (\prod_{d \in \mathbb{A}} S_{d}) c = {\rm id}.
\end{equation}
Since the product $\prod_{d \in \mathbb{A}} S_{d}$ contains both (nontrivial) upper and lower triangular matrices, any representation corresponding a tuple $(c,h,(S_{d})_{d \in \mathbb{A}})$ satisfying the relation \eqref{eq_rela_exmp} is irreducible. Then we have the following proposition:

\begin{prop}\label{prop_ram_sl2}
    Let $G={\rm SL}_2(\mathbb{C})$. Given a ramified irregular type $Q$, we have
    \begin{align*}
        \mathcal{M}_{\rm B}(X_{\boldsymbol{D}},{\rm SL}_2(\mathbb{C}) ,Q, \theta) \cong \mathcal{M}_{\rm B}(X_{\boldsymbol{D}},{\rm SL}_2(\mathbb{C}) ,Q, 0)
    \end{align*}
    for any weight $\theta$.
\end{prop}

\begin{rem}\label{rem_ram_sln}
In the case of ${\rm SL}_n(\mathbb{C})$, if the leading coefficient of a ramified irregular type $Q$ is regular and semisimple, the same consequence holds as Proposition \ref{prop_ram_sl2}, i.e.
\begin{align*}
    \mathcal{M}_{\rm B}(X_{\boldsymbol{D}},{\rm SL}_n(\mathbb{C}) ,Q, \theta) \cong \mathcal{M}_{\rm B}(X_{\boldsymbol{D}},{\rm SL}_2(\mathbb{C}) ,Q, 0)
\end{align*}
for any weight $\theta$.
\end{rem}

\section{Relation to Wild Nonabelian Hodge Correspondence on Curves}\label{sect_nahc}

The authors established the wild nonabelian Hodge correspondence for principal bundles on curves at the level of categories. Based on the existence of the corresponding moduli spaces, the Dolbeault moduli space admits a hyperK\"alher structure. Since the wild nonabelian Hodge correspondence for principal bundles is only given when irregular types are unramfied, we always assume that the irregular types considered in this section are unramified.

\subsection{Betti Moduli Spaces}

We follow the construction given in \S\ref{subsect_cons_pi''} and define a map
\begin{align*}
    {\rm Hom}_{\mathbb{S}}(\Omega'',\boldsymbol{P}) \rightarrow \prod_{x \in \boldsymbol{D}} P_{-\theta_x}
\end{align*}
as
\begin{align*}
    ((a''_i,b''_i)_{1 \leq i \leq g} , (l''_x, h''_x, S''_{x,d})_{x \in \boldsymbol{D},d \in \mathbb{A}_x}  ) \rightarrow (l''_x h''_x)_{x \in \boldsymbol{D}}.
\end{align*}
We fix a collection of elements $M_{\boldsymbol\theta} = \{ M_{\theta_x} , x \in \boldsymbol{D} \}$ in the Levi subgroups, i.e. $M_{\theta_x} \in L_{\theta_x}$, which is regarded as an element in $\prod_{x \in \boldsymbol{D}} L_{\theta_x}$. Consider the following composition of maps
\begin{align*}
    \widetilde{{\rm Hom}}_{\mathbb{S}}(\Omega'', \boldsymbol{P}) \rightarrow {\rm Hom}_{\mathbb{S}}(\Omega'',\boldsymbol{P}) \rightarrow \prod_{x \in \boldsymbol{D}} P_{-\theta_x} \rightarrow \prod_{x \in \boldsymbol{D}} L_{\theta_x}.
\end{align*}
Denote by $\widetilde{{\rm Hom}}_{\mathbb{S}}(\Omega'', \boldsymbol{P}, M_{\boldsymbol\theta}) \subseteq \widetilde{{\rm Hom}}_{\mathbb{S}}(\Omega'', \boldsymbol{P})$ the preimage of $M_{\boldsymbol\theta} \in \prod_{x \in \boldsymbol{D}} L_{\theta_x}$. Taking the restriction
\begin{align*}
    {\rm Hom}_{\mathbb{S}}(\Omega'', \boldsymbol{\theta}, M_{\boldsymbol\theta}) : = \widetilde{{\rm Hom}}_{\mathbb{S}}(\Omega'', \boldsymbol{P}, M_{\boldsymbol\theta}) |_{ {\rm Hom}_{\mathbb{S}}(\Omega'', G)  }
\end{align*}
and adding the relation \eqref{eq_rela_Pi''}, we obtain a variety ${\rm Hom}_{\mathbb{S}}(\Pi'', \boldsymbol{\theta}, M_{\boldsymbol\theta})$ together with an induced $(G \times \boldsymbol{L})$-action. We define
\begin{align*}
    \mathcal{M}_{\rm B}(X_{\boldsymbol{D}},G,\boldsymbol{Q},\boldsymbol\theta, M_{\boldsymbol\theta}) := {\rm Hom}_{\mathbb{S}}(\Pi'', \boldsymbol{\theta}, M_{\boldsymbol\theta}) / \!\! / (G \times \boldsymbol{L}, \chi_{\boldsymbol\theta}).
\end{align*}
As a direct result of Theorem \ref{thm_Sto_moduli}, we obtain the Betti moduli space considered in the (unramified) wild nonabelian Hodge correspondence on noncompact curves \cite{HS22}.

\begin{cor}\label{cor_Betti}
    There exists a quasi-projective variety $\mathcal{M}_{\rm B}(X_{\boldsymbol{D}},G,\boldsymbol{Q},\boldsymbol\theta, M_{\boldsymbol\theta})$ as the moduli space of degree zero $R$-semistable $\boldsymbol\theta$-filtered Stokes $G$-local systems with irregular type $\boldsymbol{Q}$ on $X_{\boldsymbol{D}}$ such that the Levi factors of formal monodromies around punctures are given by $M_{\boldsymbol\theta}$ (up to conjugation).
\end{cor}

\subsection{De Rham and Dolbeault Moduli Spaces}

Fixing a collection of weights $\boldsymbol\theta$, let $\mathcal{G}_{\boldsymbol\theta}$ be the corresponding parahoric group scheme on $X$ (see \cite[\S2]{BS15} or \cite[\S2.1]{KSZ23} for instance). We fix a collection of unramified irregular types $\boldsymbol{Q}$. The moduli space $\mathcal{M}_{\rm Higgs}(X,\mathcal{G}_{\boldsymbol\theta},\boldsymbol{Q})$ (resp. $\mathcal{M}_{\rm Conn}(X,\mathcal{G}_{\boldsymbol\theta},\boldsymbol{Q})$) of degree zero $R$-semistable merohoric (= meromorphic and parahoric) $\mathcal{G}_{\boldsymbol\theta}$-torsors with irregular type $\boldsymbol{Q}$ (resp. degree zero $R$-semistable meromorphic parahoric $\mathcal{G}_{\boldsymbol\theta}$-connections with irregular type $\boldsymbol{Q}$) exists. The approach of constructing these two moduli spaces is the same as that for logahoric (= logarithmic and parahoric) $\mathcal{G}_{\boldsymbol\theta}$-torsors \cite[\S 6]{KSZ23}, where the authors constructed the moduli spaces of logahoric Higgs torsors via an equivalence between logarithmic Higgs torsors on $X$ and equivariant $G$-Higgs bundles on an appropriate cover. We also want to remind the reader that it needs a more careful discussion for the corresponding Dolbeault and de Rham moduli spaces when irregular types are ramified because there are two covers needed to be taken into consideration: one is for weights and the other is for the irregular types.

Given the existence for the moduli spaces of merohoric Higgs torsors and merohoric connections, we define the Dolbeault and de Rham residue morphisms in order to obtain the desired Dolbeault and de Rham moduli spaces in the wild nonabelian Hodge correspondence \cite[Theorem in \S1]{HS22}, and a similar construction in the parabolic case is given in \cite[\S 7]{BGM20}. We take the Dolbeault side as an example. We define the following map
\begin{align*}
    \mu_{\mathrm{Dol}}: \mathcal{M}_{\mathrm{Higgs}}(X,\mathcal{G}_{\boldsymbol{\theta}},\boldsymbol{Q})&\longrightarrow\prod_{x\in\boldsymbol{D}}(\mathfrak{h}_x\cap\mathfrak{l}_{\theta_x})/L_{\theta_x}\\
     (\mathcal{E},\varphi)
     &\longmapsto\prod_{x\in\boldsymbol{D}}L_{\theta_x}\cdot[\mathrm{Res}_x(\varphi)],
\end{align*}
where $\mathfrak{h}_x$ is the Lie algebra of the stabilizer $H_x$ and $\mathfrak{l}_{\theta_x}$ is the Lie algebra of the Levi subgroup $L_{\theta_x}$. Given a $L_{\boldsymbol{\theta}}$-orbit $\boldsymbol{\mathcal{O}}=\{\mathcal{O}_x, x\in\boldsymbol{D}\}\subseteq \prod_{x\in\boldsymbol{D}}(\mathfrak{h}_x\cap\mathfrak{l}_{\theta_x})/L_{\theta_x}$, the fiber
\begin{align*}
    \mathcal{M}_{\mathrm{Dol}}(X,\mathcal{G}_{\boldsymbol{\theta}},\boldsymbol{Q},\boldsymbol{\mathcal{O}}):=\mu_{\mathrm{Dol}}^{-1}(\boldsymbol{\mathcal{O}})
\end{align*}
is the moduli space of degree zero $R$-semistable meromorphic parahoric $\mathcal{G}_{\boldsymbol{\theta}}$-Higgs torsors with residue data lying inside $\boldsymbol{\mathcal{O}}$. If $\boldsymbol{\mathcal{O}}$ corresponds to the $L_{\boldsymbol{\theta}}$-orbit of a collection of Levi factors $\varphi_{\boldsymbol{\theta}}$, by abuse of notation, the moduli space $\mathcal{M}_{\mathrm{Dol}}(X,\mathcal{G}_{\boldsymbol{\theta}},\boldsymbol{Q},\boldsymbol{\mathcal{O}})$ is also denoted as $\mathcal{M}_{\mathrm{Dol}}(X,\mathcal{G}_{\boldsymbol{\theta}},\boldsymbol{Q},\varphi_{\boldsymbol{\theta}})$.

As a direct result of the wild nonabelian Hodge correspondence at the level of categories \cite[Theorem in \S 1]{HS22}, we formulate the wild nonabelian Hodge correspondence at the level of moduli spaces.

\begin{thm}\label{thm_nahc}
    Let $(\boldsymbol{\alpha}, \varphi_{\boldsymbol{\alpha}}, \widetilde{\boldsymbol{Q}})$, $ (\boldsymbol{\beta},\nabla_{\boldsymbol{\beta}}, \boldsymbol{Q})$, and $(\boldsymbol{\gamma}, M_{\boldsymbol{\gamma}}, \boldsymbol{Q})$ be the local data for Dolbeault, de Rham and Betti, respectively. For each $x\in\boldsymbol{D}$, these data are grouped into the following table,
    \begin{center}
	\begin{tabular}{|c|c|c|c|}
		\hline
		\rule{0pt}{2.6ex} & \ \ \ Dolbeault\ \ \ & \ \ \ de Rham\ \ \ & \ \ \ Betti\ \ \ \\[0.5ex]
		\hline
		\rule{0pt}{2.6ex} weights & $\alpha_x$ & $\beta_x$ & $\gamma_x$ \\[0.5ex]
		\hline
		\rule{0pt}{2.6ex} residues $\backslash$ monodromies & $\varphi_{\alpha_x}$ & $\nabla_{\beta_x}$ & $M_{\gamma_x}$ \\[0.7ex]
		\hline
		\rule{0pt}{2.6ex} irregular types & $\widetilde{Q}_x$ & $Q_x$ & $Q_x$ \\[0.5ex]
		\hline
	\end{tabular}
\end{center}
Let $\nabla_{\beta_x} = s_{\beta_x} + Y_{\beta_x}$ be the Jordan decomposition, where $s_{\beta_x}$ is the semisimple part and $Y_{\beta_x}$ is the nilpotent part. We complete $Y_{\beta_x}$ into an $\mathfrak{sl}_2$-triple $(X_{\beta_x}, H_{\beta_x}, Y_{\beta_x})$. Suppose these data are subjected into the relations in the following table,
    \begin{center}
	\begin{tabular}{|c|c|c|c|}
		\hline
		\rule{0pt}{2.6ex} & Dolbeault & \ \ de Rham\ \ & Betti \\[0.5ex]
		\hline
		\rule{0pt}{2.6ex} weights & $\frac{1}{2}(s_{\beta_x} + \bar{s}_{\beta_x})$ & $\beta_x$ & $\beta_x - \frac{1}{2}(s_{\beta_x} + \bar{s}_{\beta_x})$ \\[0.7ex]
		\hline
		\rule{0pt}{2.6ex} residues $\backslash$ monodromies & $\frac{1}{2}(s_{\beta_x} - \beta_x) + (Y_{\beta_x} - H_{\beta_x} + X_{\beta_x})$ & $s_{\beta_x} + Y_{\beta_x}$ & $\exp(-2\pi \sqrt{-1}(s_{\beta_x} + Y_{\beta_x}))$ \\[0.7ex]
		\hline
		\rule{0pt}{2.6ex} irregular types & $\widetilde{Q}_x=\frac{1}{2}Q_x$ & $Q_x$ & $Q_x$ \\[0.7ex]
		\hline
	\end{tabular}
\end{center}
Then the irregular Riemann--Hilbert correspondence (Theorem \ref{thm_global_con_sys}) gives rise to an isomorphism of complex analytic spaces
\begin{align*}
    \mathcal{M}_{\mathrm{B}}^{\mathrm{(an)}}(X_{\boldsymbol{D}},G,\boldsymbol{Q},\boldsymbol{\gamma},M_{\boldsymbol{\gamma}})\cong \mathcal{M}_{\mathrm{dR}}^{\mathrm{(an)}}(X,\mathcal{G}_{\boldsymbol{\beta}},\boldsymbol{Q},\nabla_{\boldsymbol{\beta}}),
\end{align*}
and we also have a homeomorphism of topological spaces
\begin{align*}
    \mathcal{M}_{\mathrm{Dol}}^{\mathrm{(top)}}(X,\mathcal{G}_{\boldsymbol{\alpha}},\widetilde{\boldsymbol{Q}},\varphi_{\boldsymbol{\alpha}})\cong \mathcal{M}_{\mathrm{dR}}^{\mathrm{(top)}}(X,\mathcal{G}_{\boldsymbol{\beta}},\boldsymbol{Q},\nabla_{\boldsymbol{\beta}}).
\end{align*}
\end{thm}

As a direct application, the Dolbeault moduli space (and the de Rham and Betti moduli spaces as well) is hyperK\"ahler, we will sketch a proof via wild harmonic principal bundles, weighted Sobolev spaces, and hyperK\"ahler reduction. The hyperK\"ahler geometry for the moduli spaces in the classical nonabelian Hodge theory has been well-studied decades ago, for example, \cite{Hit87,Fuj91}. A similar property is also known for the tame and wild cases (for $\mathrm{GL}_n(\mathbb{C})$ or $\mathrm{SL}_n(\mathbb{C})$ as the structure group) from the work of Konno, Nakajima, Biquard, Biquard--Boalch, among others \cite{Kon93,Nak96,Biq97,BB04}. We would like to mention that, besides the hyperK\"ahler property, the completeness of hyperK\"ahler metrics was also studied by Biquard--Boalch \cite{BB04}.

\begin{thm}\label{hK-Higgs}
$\mathcal{M}_{\mathrm{Dol}}(X,\mathcal{G}_{\boldsymbol{\theta}},\boldsymbol{Q}, \varphi_{\boldsymbol\theta})$ admits a hyperK\"ahler structure.
\end{thm}

\begin{proof}[Sketch of the proof]
To prove the theorem, here we first introduce some notions from analytic aspects, more details can be found in \cite[\S4]{HKSZ22} and \cite[\S2.2]{HS22}.

Let $\boldsymbol{n}=\{n_x, \, x\in\boldsymbol{D}\}$ be a collection of integers labeled by points in $\boldsymbol{D}$. Given a metrized $G$-Higgs bundle $(E,\bar\partial_E,\phi,h)$ on $X_{\boldsymbol{D}}$, it is called \emph{$(\boldsymbol{\theta},\boldsymbol{n})$-adapted} if $h$ is a $\boldsymbol{\theta}$-adapted hermitian metric (see e.g. \cite{HKSZ22,HS22} for the definition of adaptedness), and for each $x\in\boldsymbol{D}$, we have
\begin{align*}
        z^{\theta_x}\cdot(z^{n_x}\cdot\phi(z))\cdot z^{-\theta_x}\ \text{is bounded as}\ z\ \text{approaches}\ 0,
\end{align*}
where $z$ is the local coordinate vanishing at $x$. In this case, $h$ induces an extension of $(E,\bar\partial_E,\phi)$ into a merohoric $\mathcal{G}_{\boldsymbol{\theta}}$-Higgs torsor $(\mathcal{E},\varphi)$ on $X$, with $\varphi\in H^0(X,\mathcal{E}(\mathfrak{g})\otimes K_X(\boldsymbol{nD}))$, where $\boldsymbol{nD}:=\sum_{x \in \boldsymbol{D}} n_x \cdot x$.

%{\color{red}(this is for $Q$ is ramified case?)It is further called \emph{$(\boldsymbol{Q},\boldsymbol{\theta})$-adapted} if up to a Galois cover of ramification index $d_x$ (i.e. the smallest denominator of $Q_x(z)$) at $x\in\boldsymbol{D}$, it is $(\boldsymbol{n},\boldsymbol{\theta})$-adapted with each $n_x=d_x\cdot\mathrm{deg}(Q_x)+1$, and the extended  merohoric $\mathcal{G}_{\boldsymbol{\theta}}$-Higgs torsor $(\mathcal{E},\varphi)$ is of irregular type $\boldsymbol{Q}$.  }

Now we give an analytic proof via wild harmonic bundles. Let $\mathcal{E}$ be a parahoric $\mathcal{G}_{\boldsymbol{\theta}}$-torsor on $X$, and put $E:=\mathcal{E}|_{X_{\boldsymbol{D}}}$, which is a $G$-bundle on $X_{\boldsymbol{D}}$. We fix a background $(\boldsymbol{\theta},\boldsymbol{n})$-adapted  $G$-Higgs bundle $(E,\bar\partial_0,\phi_0,h)$ on $X_{\boldsymbol{D}}$ with (Levi factor of) residue data lying in $\boldsymbol{\mathcal{O}}$. For each $i\geq 0$, let $d_0:=\partial_h'+\bar\partial_0$ be the Chern connection.  Denote by $\mathcal{A}^i(E(\mathfrak{g}))$ the space of $C^\infty$-sections of $E(\mathfrak{g}) \otimes \Omega^i_{X_{\boldsymbol{D}}}$, and let $D^p_k\mathcal{A}^i(E(\mathfrak{g}))$ be the Sobolev completion of $\mathcal{A}^i(E(\mathfrak{g}))$ with respect to the Sobolev norm $\|\bullet\|_{D^p_k}$ (see \cite{Kon93} for details).

Note that the metric $h$ is a section of $K$-reduction of $E$, i.e. $h\in\mathcal{A}^0(E/K)$, where $K\subseteq G$ is the maximal compact subgroup. Let $\mathfrak{k}$ be the Lie algebra of $K$, and then we have the Cartan decomposition $\mathfrak{g}=\mathfrak{k}\oplus i\mathfrak{k}$, which induces the decomposition
\begin{align*}
        E(\mathfrak{g})=E(\mathfrak{h})\oplus iE(\mathfrak{h}).
\end{align*}
Let $F=h^*E$ be the $K$-reduction, which is a $K$-bundle with $F\times_KG\cong E$. Define
\begin{align*}
        \mathbb{G}&=\{ g\in D_2^p\mathcal{A}^0(E(G)):\ g\in\mathcal{A}^0(F(K))\},\\
        \mathbb{A}&=\{(d_A,\phi): d_A-d_0\in D^p_1\mathcal{A}^1(E(\mathfrak{h})), \phi-\phi_0\in D_1^p\mathcal{A}^{1,0}(E(\mathfrak{g}))\},
\end{align*}
there is an adjoint action of the unitary gauge group $\mathbb{G}$ on $\mathbb{A}$. The affine space $\mathbb{A}$ is of infinite dimensional, and it is hyperK\"ahler. Indeed, its tangent space can be identified with $D^p_1\mathcal{A}^{0,1}(E(\mathfrak{g}))\oplus D_1^p\mathcal{A}^{1,0}(E(\mathfrak{g}))$ with Riemannian metric given by
\begin{align*}
        g((\xi_1,\eta_1),(\xi_2,\eta_2))=i\int_{X_{\boldsymbol{D}}}\mathrm{Tr}(\eta_2\wedge\xi_1-\eta_1\wedge\xi_2).
\end{align*}
The three complex structures are
\begin{align*}
        I(\xi,\eta)=(i\xi,i\eta),\ J(\xi,\eta)=(i\bar{\eta}^{\mathrm{T}},-i\bar{\xi}^{\mathrm{T}}),\ K(\xi,\eta)=IJ(\xi,\eta)=(-\bar{\eta}^{\mathrm{T}},\bar{\xi}^{\mathrm{T}}).
\end{align*}
The action of $\mathbb{G}$ on $\mathbb{A}$ admits a hyperK\"ahler moment map $\boldsymbol{\mu}=(\mu_I,\mu_J,\mu_K)$:
    \begin{align}\label{wild_harmonic}
    \left\{
    \begin{aligned}
         \mu_I(d_A,\phi)&=F_A+[\phi,\phi^*],\\
        (\mu_J+i\mu_K)(d_A,\phi)&=2i\bar\partial_A\phi.
    \end{aligned}
    \right.
    \end{align}
    A solution to \eqref{wild_harmonic} is exactly a wild harmonic bundle on $X_{\boldsymbol{D}}$. Then the hyperK\"ahler quotient $\boldsymbol{\mu}^{-1}(0)/\mathbb{G}=(\mu_I^{-1}(0)\cap\mu_J^{-1}(0)\cap\mu_K^{-1}(0))/\mathbb{G}$ admits a hyperK\"ahler structure, as the moduli space of $(\boldsymbol{\theta},\boldsymbol{n})$-adapted  wild harmonic bundles on $X_{\boldsymbol{D}}$ with (Levi factor of) residue data lying in $\boldsymbol{\mathcal{O}}$, which is naturally diffeomorphic to $\mathcal{M}_{\mathrm{Dol}}(X,\mathcal{G}_{\boldsymbol{\theta}},\boldsymbol{Q},\varphi_{\boldsymbol\theta})$.

\end{proof}

\bibliographystyle{amsalpha}
\bibliography{ref_Betti}

\bigskip
\noindent\small{\textsc{Max Planck Institute for Mathematics in the Sciences}\\
		 Inselstraße 22, 04103 Leipzig, Germany}\\
\emph{E-mail address}:  \texttt{pfhwangmath@gmail.com}

\bigskip
\noindent\small{\textsc{Department of Mathematics, South China University of Technology}\\
		381 Wushan Rd, Tianhe Qu, Guangzhou, Guangdong, China}\\
\emph{E-mail address}:  \texttt{hsun71275@scut.edu.cn}

\end{document}